\documentclass[a4paper]{article}
\usepackage[utf8]{inputenc}
\usepackage[english]{babel}
\usepackage[T1]{fontenc}
\usepackage{amsmath,amssymb,bm}

\usepackage{hyperref}

\renewcommand{\bar}{\overline}

\renewcommand{\tilde}{\widetilde}

\usepackage{microtype}

\usepackage{xspace}

\usepackage[all,line]{xy}

\usepackage{mathtools}

\newcommand{\spa}[2]{\operatorname{span}_{#1}\ensuremath{\left\{#2\right\}}}

\newcommand{\inv}{^{-1}}

\newcommand{\iso}{\cong}

\DeclareMathOperator{\Tr}{Tr}

\DeclareMathOperator{\End}{End}

\DeclareMathOperator{\wt}{wt}

\DeclareMathOperator{\Hom}{Hom}

\DeclareMathOperator{\id}{id}

\DeclareMathOperator{\res}{res}

\DeclareMathOperator{\Ann}{Ann}

\DeclareMathOperator{\ad}{ad}

\newcommand{\Suppess}{\ensuremath{\operatorname{Supp}_{\operatorname{ess}}}}

\newcommand{\Supp}{\ensuremath{\operatorname{Supp}}}

\newcommand{\tensor}{\otimes}

\renewcommand{\phi}{\varphi}

\renewcommand{\epsilon}{\varepsilon}

\newcommand{\setC}{\mathbb{C}}

\newcommand{\setN}{\mathbb{N}}

\newcommand{\setZ}{\mathbb{Z}}

\newcommand{\setQ}{\mathbb{Q}}

\usepackage[amsmath,thmmarks,hyperref]{ntheorem}

\newtheorem{thm}{Theorem}[section]

\newtheorem{prop}[thm]{Proposition}

\newtheorem{lemma}[thm]{Lemma}

\newtheorem{cor}[thm]{Corollary}

\newtheorem{defn}[thm]{Definition}

\theoremstyle{nonumberplain}
\theoremheaderfont{%
\normalfont\bfseries}
\theorembodyfont{\normalfont}
\theoremsymbol{\ensuremath{\square}}
\theoremseparator{.}
\newtheorem{proof}{Proof}

\usepackage{enumitem}

\usepackage[final]{fixme}

\begin{document}

\title{Irreducible quantum group modules with finite dimensional
  weight spaces. I} \author{Dennis Hasselstrøm Pedersen} \date{}

\maketitle

\begin{abstract}
  In this paper we classify all simple weight modules for a quantum
  group $U_q$ at a complex odd root of unity $q$ when the Lie algebra
  is not of type $G_2$. By a weight module we mean a finitely
  generated $U_q$-module which has finite dimensional weight spaces
  and is a sum of those. Our approach follows the procedures used by
  S. Fernando~\cite{Fernando} and O. Mathieu~\cite{Mathieu} to solve
  the corresponding problem for semisimple complex Lie algebras.
\end{abstract}


\section{Introduction and notation}
\label{sec:intr-notat}
Let $\mathfrak{g}$ be a simple complex Lie algebra not of type
$G_2$. Let $q\in \setC$ be a nonzero element and let
$U_q:=U_q(\mathfrak{g})$ be the quantum group over $\setC$ with $q$ as
the quantum parameter (defined below). We want to classify all simple
weight modules for $U_q$. In the papers~\cite{Fernando}
and~\cite{Mathieu} this is done for $\mathfrak{g}$-modules. Fernando
proves in the paper~\cite{Fernando} that the classification of simple
$\mathfrak{g}$ weight modules essentially boils down to classifying
two classes of simple modules: The finite dimensional simple modules
and the so called 'torsion free' simple modules. The classification of
finite dimensional modules is well known in the classical case (as
well as in the quantum group case) so the remaining problem is to
classify the torsion free simple modules. Olivier Mathieu classifies
these in the classical case in~\cite{Mathieu}. The classification uses
the concept of $\mathfrak{g}$ coherent families which are huge
$\mathfrak{g}$ modules with weight vectors for every possible weight,
see~\cite[Section~4]{Mathieu}. Mathieu shows that every torsion free
simple module is a submodule of a unique irreducible semisimple
coherent family and each of these irreducible semisimple coherent
families contains a so-called admissible simple highest weight module
as well. This reduces the classification to the classification of
admissible simple highest weight modules.

\subsection{Main results}
\label{sec:main-results}
In this paper we will first carry out the reduction done by Fernando
to the quantum group case for $q$ a non-root-of-unity and $q$ an odd
root of unity. Then we carry out the classification of torsion free
simple module in the root of unity case. The corresponding
classification of torsion free simple modules for generic $q$ turns
out to be much harder. We leave this to a subsequent
paper~\cite{DHP2}.

We will follow closely the methods described in the two above
mentioned papers. Many of the results can be directly translated from
the classical case but in several cases we have to approach the
problem a little differently. One of the first differences we
encounter is the fact that in~\cite{Fernando} concepts are defined by
using the root system without first choosing a base. Then later a base
is chosen in an appropiate way. In the quantum group case we define
the quantized enveloping algebra by first choosing a base of the root
system and then defining the simple root vectors $E_\alpha$,
$F_\alpha$, etc. This means that we can't later change the basis like
in~\cite{Fernando}. The solution is to consider 'twists' of modules by
Weyl group elements cf. definition~\ref{def:2}. Another difference is
the fact that we do not a priori have root vectors $E_\beta$ for any
positive root $\beta$ unless $\beta$ is simple. Root vectors can be
constructed but the construction involves a choice of a reduced
expression for the longest element of the Weyl group $w_0$. The root
vectors constructed depend on this choice. So if we want to use root
vectors to define our terms we should prove that our definitions are
independent of the choice of the root vectors. Once the root vectors
are defined we continue like in the classical case with some
differences. Notably the proof of Proposition~\ref{prop:5} is
different. Here we reduce the problem to rank $2$ calculations in the
quantized enveloping algebra. This is also the main reason we exclude
$\mathfrak{g}$ of type $G_2$ in this paper.

In the root of unity case the classification of simple weight modules
reduces completely to the classical case as seen in
Section~\ref{Torsion_free_modules_root_of_unity}. We use the same
procedure as in~\cite{Mathieu} to reduce the problem to classifying
coherent families and then we show that all irreducible coherent
families in the root of unity case can be constructed via classical
$\mathfrak{g}$ coherent families.

\subsection{Acknowledgements}
I would like to thank my advisor Henning H. Andersen for great
supervision and many helpful comments and discussions and Jacob
Greenstein for introducing me to this problem when I was visiting him
at UC Riverside in the fall of 2013. The authors research was
supported by the center of excellence grant 'Center for Quantum
Geometry of Moduli Spaces' from the Danish National Research
Foundation (DNRF95).

\subsection{Notation}
\label{sec:notation}
We will fix some notation: We denote by $\mathfrak{g}$ a fixed simple
Lie algebra over the complex numbers $\setC$. We assume $\mathfrak{g}$
is not of type $G_2$ to avoid unpleasant computations.

Fix a triangular decomposition of $\mathfrak{g}$: Let $\mathfrak{h}$
be a maximal toral subalgebra and let $\Phi \subset \mathfrak{h}^*$ be
the roots of $\mathfrak{g}$ relative to $\mathfrak{h}$. Choose a
simple system of roots $\Pi = \{\alpha_1,\dots,\alpha_n\} \subset
\Phi$. Let $\Phi^+$ (resp. $\Phi^-$) be the positive (resp. negative)
roots. Let $\mathfrak{g}^{\pm}$ be the positive and negative part of
$\mathfrak{g}$ corresponding to the simple system $\Pi$. So
$\mathfrak{g} = \mathfrak{g}^- \oplus \mathfrak{h} \oplus
\mathfrak{g}^+$. Let $W$ be the Weyl group generated by the simple
reflections $s_i := s_{\alpha_i}$. For a $w\in W$ let $l(w)$ be the
length of $W$ i.e. the smallest amount of simple reflections such that
$w=s_{i_1}\cdots s_{i_{l(w)}}$. Let $(\cdot|\cdot)$ be a standard
$W$-invariant bilinear form on $\mathfrak{h}^*$ and
$\left<\alpha,\beta^\vee\right> =
\frac{2(\alpha|\beta)}{(\beta|\beta)}$. Since $(\cdot|\cdot)$ is
standard we have $(\alpha|\alpha)=2$ for any short root $\alpha\in
\Phi$. Let $Q=\spa{\setZ}{\alpha_1,\dots,\alpha_n}$ denote the root
lattice and $\Lambda=\spa{\setZ}{\omega_1,\dots,\omega_n}\subset
\mathfrak{h}^*$ the integral lattice where $\omega_i\in
\mathfrak{h}^*$ are the fundamental weights defined by
$(\omega_i|\alpha_j)=\delta_{ij}$.

Let $U_v=U_v(\mathfrak{g})$ be the corresponding quantized enveloping
algebra defined over $\mathbb{Q}(v)$ as defined in~\cite{Jantzen} with
generators $E_\alpha,F_\alpha,K_\alpha^{\pm 1}$, $\alpha\in\Pi$ and
certain relations which can be found in Chapter~4
of~\cite{Jantzen}. We define $v_\alpha = v^{(\alpha|\alpha)/2}$
(i.e. $v_\alpha = v$ if $\alpha$ is a short root and $v_\alpha = v^2$
if $\alpha$ is a long root) and for $n\in\setZ$, $[n]_v
=\frac{v^n-v^{-n}}{v-v\inv}$.  Let $[n]_\alpha := [n]_{v_\alpha} =
\frac{v_\alpha^n-v_\alpha^{-n}}{v_\alpha-v_\alpha\inv}$. We omit the
subscripts when it is clear from the context. For later use we also
define the quantum binomial coefficients: For $r\in \setN$ and $a\in
\setZ$:
\begin{equation*}
  {a \brack r}_v = \frac{[a][a-1]\cdots [a-r+1]}{[r]!}
\end{equation*}
where $[r]! := [r][r-1]\cdots [2][1]$. Let $A=\setZ[v,v\inv]$ and let
$U_A$ be Lusztigs $A$-form defined in~\cite{MR1066560}, i.e. the $A$
subalgebra generated by the divided powers
$E_\alpha^{(n)}:=\frac{1}{[n]_\alpha!}E_\alpha^{n}$,
$F_\alpha^{(n)}:=\frac{1}{[n]_\alpha!}F_\alpha^{n}$ and $K_\alpha^{\pm
  1}$, $\alpha\in\Pi$.

Let $q\in \setC$ be a nonzero complex number and set $U_q = U_A
\tensor_A \setC_q$ where $\setC_q$ is the $A$-module equal to $\setC$
as a vector space where $v$ is sent to $q$. In the following sections
we will distinguish between whether $q$ is a root of unity or not.

We have a triangular decomposition of Lusztigs $A$-form $U_A = U_A^-
\tensor U_A^0 \tensor U_A^+$ with $U_A^-$ the $A$ subalgebra generated
by $\{F_\alpha^{(n)}|\alpha\in \Pi,n\in \setN\}$ in $U_A$, $U_A^+$ the
$A$ subalgebra generated by $\{E_\alpha^{(n)}|\alpha\in \Pi,n\in \setN
\}$ in $U_A$ and $U_A^0$ the $A$ subalgebra generated by
$\{K_\alpha^{\pm 1}, { K_\alpha ; c \brack r}|\alpha\in \Pi, c\in
\setZ, r\in \setN\}$ in $U_A$ where
\begin{equation*}
  {K_\alpha ; c \brack r} := \prod_{j=1}^r \frac{ K_\alpha v_\alpha^{c+1-j}-K_\alpha\inv v_\alpha^{-c-1+j}}{v_\alpha^{j}-v_\alpha^{-j}}.
\end{equation*}
For later use we also define $[K_\alpha; r]={K_\alpha;r\brack 1}$. We
have the corresponding triangular decomposition of $U_q$: $U_q = U_q^-
\tensor U_q^0 \tensor U_q^+$ with $U_q^{\pm} = U_A^{\pm} \tensor_A
\setC_q$ and $U_q^0 = U_A^0 \tensor_A \setC_q$.

For a $q\in \setC^*=\setC\backslash\{0\}$ define ${a \brack r}_q$ as
the image of ${a \brack r}_v$ in $\setC$.  We will omit the subscript
from the notation when it is clear from the context. We define
$q_\beta\in \setC$ and $[n]_\beta\in \setC$ as the image of
$v_\beta\in A$ and $[n]_\beta\in A$, respectively abusing
notation. Similarly, we will abuse notation and write ${K_\alpha ; c
  \brack r}$ also for the image of ${K_\alpha ; c \brack r}\in U_A$ in
$U_q$. Define for $\mu\in Q$, $K_\mu = \prod_{i=1}^n
K_{\alpha_i}^{a_i}$ if $\mu = \sum_{i=1}^n a_i \alpha_i$ with $a_i\in
\setZ$.

There is a braid group action on $U_v$ which we will describe now. We
use the definition from~\cite[Chapter~8]{Jantzen}. The definition is
slightly different from the original in~\cite[Theorem~3.1]{MR1066560}
(see \cite[Warning~8.14]{Jantzen}). For each simple reflection $s_i$
there is a braid operator that we will denote by $T_{s_i}$ satisfying
the following: $T_{s_i}:U_v\to U_v$ is a $\setQ(v)$ automorphism and for
$i\neq j \in \{1,\dots,n\}$
\begin{align*}
  T_{s_i}(K_\mu)=&K_{s_i(\mu)}
  \\
  T_{s_i}(E_{\alpha_i}) =& -F_{\alpha_i}K_{\alpha_i}
  \\
  T_{s_i}(F_{\alpha_i})=& - K_{\alpha_i}\inv E_{\alpha_i}
  \\
  T_{s_i}(E_{\alpha_j})=&
  \sum_{i=0}^{-\left<\alpha_j,\alpha_i^\vee\right>} (-1)^i
  v_{\alpha_i}^{-i} E_{\alpha_i}^{(r-i)}E_{\alpha_j}E_{\alpha_i}^{(i)}
  \\
  T_{s_i}(F_{\alpha_j})=&
  \sum_{i=0}^{-\left<\alpha_j,\alpha_i^\vee\right>} (-1)^i
  v_{\alpha_i}^{i} F_{\alpha_i}^{(i)}F_{\alpha_j}F_{\alpha_i}^{(r-i)}.
\end{align*}
The inverse $T_{s_i}\inv$ is given by conjugating with the
$\setQ$-algebra anti-automorphism $\Psi$
from~\cite[section~1.1]{MR1066560} defined as follows:
\begin{align*}
  \Psi(E_{\alpha_i}) = E_{\alpha_i}, \quad \Psi(F_{\alpha_i}) =
  F_{\alpha_i}, \quad \Psi(K_{\alpha_i}) = K_{\alpha_i}\inv, \quad
  \Psi(v) = v.
\end{align*}
The braid operators $T_{s_i}$ satisfy braid relations so we can define
$T_w$ for any $w\in W$: Choose a reduced expression of $w$:
$w=s_{i_1}\cdots s_{i_n}$. Then $T_w = T_{s_{i_1}}\cdots T_{s_{i_n}}$
is independent of the chosen reduced expression
by~\cite[Theorem~3.2]{MR1066560}. We have
$T_w(K_\mu)=K_{w(\mu)}$. Furthermore $T_w$ restricts to an
automorphism $T_w:U_A\to U_A$.

Let $w_0$ be the longest element in $W$ and let $s_{i_1}\cdots
s_{i_N}$ be a reduced expression of $w_0$. We define root vectors
$E_\beta$ and $F_\beta$ for any $\beta\in \Phi^+$ by the following:
First of all set
\begin{equation*}
  \beta_{j} = s_{i_1}\cdots s_{i_{j-1}}(\alpha_{i_j}), \, \text{ for } i=1,\dots,N.
\end{equation*}
Then $\Phi^+ = \{\beta_1,\dots,\beta_N\}$. Set
\begin{equation*}
  E_{\beta_j} = T_{s_{i_1}}\cdots T_{s_{i_{j-1}}}(E_{\alpha_{i_j}})
\end{equation*}
and
\begin{equation*}
  F_{\beta_j} = T_{s_{i_1}}\cdots T_{s_{i_{j-1}}}(F_{\alpha_{i_j}}).
\end{equation*}
In this way we have defined root vectors for each
$\beta\in\Phi^+$. These root vectors depend on the reduced expression
chosen for $w_0$ above. For a different reduced expression we might
get different root vectors. It is a fact that if $\beta\in\Pi$ then
the root vectors $E_\beta$ and $F_\beta$ defined above are the same as
the generators with the same notation
(cf. e.g.~\cite[Proposition~8.20]{Jantzen}) so the notation is not
ambigious in this case. By ``Let $E_\beta$ be a root vector'' we just
mean a root vector constructed as above for some reduced expression of
$w_0$.

\subsection{Basic definitions}
\label{sec:basic-definitions}
\begin{defn}
  Let $M$ be a $U_q$-module and $\lambda: U_q^0 \to \setC$ a character
  (i.e. an algebra homomorphism into $\setC$). Then the weight space
  $M_\lambda$ is defined as
  \begin{equation*}
    M_\lambda = \{ m\in M | \forall u\in U_q^0, u m = \lambda(u)m\}.
  \end{equation*}
  Let $X$ denote the set of characters of $U_q^0$.  Let $\wt M$ denote
  all the weights of $M$, i.e. $\wt M = \{ \lambda\in X | M_\lambda
  \neq 0 \}$.  If $q$ is not a root of unity we define for $\mu\in
  \Lambda$ the character $q^\mu$ by $q^\mu(K_\alpha) =
  q^{(\mu|\alpha)}$ for any $\alpha\in \Pi$. We also define
  $q_\beta^\mu = q^{\frac{(\beta|\beta)}{2} \mu}$.  We say that $M$
  only has integral weights if $\mu(K_\alpha)\in \pm q_\alpha^{\setZ}$
  for any $\alpha\in \Pi$, $\mu\in \wt M$.
\end{defn}
If $q$ is not a root of unity then $U_q^0$ is isomorphic to
$\setC[X_1^{\pm 1},\dots,X_n^{\pm 1}]$ and $X$ can be identified with
$(\setC^*)^n$ by sending $\mu\in X$ to
$(\mu(K_{\alpha_1}),\dots,\mu(K_{\alpha_n}))$. When $q$ is a root of
unity the situation is a bit more complex. We will show later that
when $q$ is a root of unity $X$ can be identified with $S\times
\Lambda_l \times \mathfrak{h}^*$ where $S$ is the set of homomorphisms
$Q\to \{\pm 1\}$ and $\Lambda_l$ is a finite set depending on the
order $l$ of the root of unity.  There is an action of $W$ on $X$. For
$\lambda\in X$ define $w\lambda$ by
\begin{equation*}
  (w\lambda)(u) = \lambda(T_{w\inv}(u)).
\end{equation*}
Note that $w q^\mu = q^{w(\mu)}$.

\begin{defn}
  Let $M$ be a $U_q$-module and $w\in W$. Define the twisted module
  ${^w}M$ by the following:
  
  As a vector space ${^w}M=M$ but the action is given by twisting with
  $w\inv$: For $m\in {^w}M$ and $u \in U_q$:
  \begin{equation*}
    u\cdot m = T_{w\inv}(u)m.
  \end{equation*}
  
  We also define ${^{\bar{w}}}M$ to be the inverse twist, i.e. for
  $m\in {^{\bar{w}}}M$, $u\in U_q$:
  \begin{equation*}
    u \cdot m = T_{w\inv}\inv(u) m.
  \end{equation*}
  Hence for any $U_q$-module $M$, ${^{\bar{w}}}(^{w}M) = M = {^w}(^{\bar{w}}M)$.
\end{defn}
Note that $\wt {^w}M = w(\wt M)$ and that ${^w}(^{w'}M)\iso {^{ww'}}M$
for $w,w'\in W$ with $l(ww')=l(w)+l(w')$ because the braid operators
$T_w$ satisfy braid relations. Also ${^{\bar{w}}}(^{\bar{w'}}M) \iso
{^{\bar{w'w}}}M$.

\begin{defn}
  \label{def:1}
  We define the category $\mathcal{F}=\mathcal{F}(\mathfrak{g})$ as
  the full subcategory of $U_q-\operatorname{Mod}$ such that for every
  $M\in \mathcal{F}$ we have
  \begin{enumerate}
  \item $M$ is finitely generated as a $U_q$-module.
  \item $M = \bigoplus_{\lambda\in X} M_\lambda$ and $\dim M_\lambda <
    \infty$.
  \end{enumerate}
\end{defn}
Note that the assignment $M\mapsto {^w}M$ is an endofunctor on
$\mathcal{F}$ (in fact an auto-equivalence).

The goal of this paper is to classify all the simple modules in
$\mathcal{F}$ in the case where $q \in \setC$ is a root of unity. Our
first step is a reduction to so called torsion free simple modules,
see Definition~\ref{def:torsion-free}. This reduction actually works
for generic $q$ as well and we treat that case first, see
Section~\ref{sec:nonroot-unity-case}. Then in
Section~\ref{sec:root-unity-case} we prove the corresponding reduction
when $q$ is a root of $1$. To handle the torsion free simple modules
we need some detailed calculations - found in~\cite{DHP-twist} and
recalled in Section~\ref{sec:u_a-calculations} - on the commutation
relations among quantum root vectors. Then we prove the classification
of torsion free simple modules in
Section~\ref{Torsion_free_modules_root_of_unity} and
Section~\ref{sec:coherent-families}. The classification for generic
$q$ turns out to be somewhat harder and will be the subject of a
subsequent paper~\cite{DHP2}.

\section{Nonroot of unity case: Reduction}
\label{sec:nonroot-unity-case}
In this section we fix a non-root-of-unity $q\in\setC^*$.

\begin{defn}
  \label{def:2}
  Let $M\in\mathcal{F}$ and let $\beta$ be a root. $M$ is called
  $\beta$-finite if for all $\lambda\in \wt M$ we have that
  $q^{\mathbb{N}\beta} \lambda \cap \wt M$ is a finite set. Here
  $q^{\mathbb{N}\beta}$ is the set $\{q^{i\beta}|i\in\setN \}$ and
  $q^{i\beta}\lambda$ just means pointwise multiplication of
  characters.
\end{defn}
As an example consider a highest weight module $M$. For any positive
root $\beta\in \Phi^+$, $M$ is $\beta$-finite. If $M$ is a Verma
module then $M$ is not $\beta$-finite for any negative root
$\beta\in\Phi^-$.

\begin{prop}
  \label{prop:1}
  Let $M\in\mathcal{F}$ and $\beta$ a positive root. Let $E_\beta$ be
  any choice of a root vector corresponding to $\beta$. Then the
  following are equivalent
  \begin{enumerate}
  \item $M$ is $\beta$-finite.
  \item For all $m\in M$, $E_{\beta}^r m = 0$ for $r\gg 0$
  \end{enumerate}
\end{prop}
\begin{proof}
  Note that $E_\beta M_\lambda \subset M_{q^\beta \lambda}$.  This
  shows that $1.$ implies $2.$. Now assume $2.$ and assume $M$ is not
  $\beta$-finite. Then we must have a $\lambda\in \wt M$, an
  increasing sequence $\{j_i\}_{i\in\mathbb{N}}\subseteq \setN$,
  weights $\mu_i = q^{j_i \beta}\lambda \in \wt M$ and weight vectors
  $0\neq m_{i}\in M_{\mu_i}$ such that $E_\beta m_i = 0$. If
  $\lambda(K_\beta) = \pm q_\beta^{j}$ for some $j\in \setZ$ then we
  can asssume without loss of generality that $j\in \setN$ since
  otherwise we can replace $\lambda$ by $q^{j_i\beta}\lambda$ for some
  sufficiently large $j_i$.

  Now consider the subalgebra $D$ of $U_q$ generated by $E_\beta$,
  $K_\beta^{\pm 1}$ and $F_\beta$ where $F_\beta$ is the corresponding
  root vector to $E_\beta$ (i.e. if $E_\beta = T_w(E_{\alpha_i})$ then
  $F_\beta = T_w(F_{\alpha_i})$). This is a subalgebra isomorphic to
  $U_{q_{\beta}}(\mathfrak{sl}_2)$. For each $i$ we get a
  $U_{q_\beta}(\mathfrak{sl}_2)$-module $D m_i$ with highest weight
  $\mu_i$. We claim that in each of those modules we have a weight
  vector $v_i\in Dm_i$ of weight $\lambda$:

  To prove the claim it is enough to show that $F_\beta^{(j_i)} m_{i}
  \neq 0$ since $F_\beta$ decreases the weight by $\beta$
  (i.e. $F_\beta M_\mu \subset M_{q^{-\beta}\mu}$). To show
  this we show that $E_\beta^{(j_i)} F_\beta^{(j_i)} m_i \neq 0$. In
  the following we will use Kac's formula:
  \begin{equation*}
    E_\beta^{(r)}F_\beta^{(s)} = \sum_{j\geq 0} F_\beta^{(s-j)} { K_\beta ; 2j-r-s \brack j} E_\beta^{(r-j)}.
  \end{equation*}
  This is a well known formula that can be found in
  e.g.~\cite[Lemma~1.7]{Jantzen} (although in this reference it is
  written in a slightly different form).
  \begin{align*}
    E_\beta^{(j_i)} F_\beta^{(j_i)} m_i =& \sum_{s\geq 0}
    F_\beta^{(j_i-s)} { K_\beta ; 2s - 2j_i \brack s}
    E_\beta^{(j_i-s)} m_i
    \\
    =& {K_\beta ; 0 \brack j_i} m_i
    \\
    =& \prod_{t=1}^{j_i} \frac{q_\beta^{1-t}\mu_i(K_\beta)-
      q_\beta^{t-1} \mu_i(K_\beta)\inv}{q_\beta^t - q_\beta^{-t}} m_i
    \\
    =& \prod_{t=1}^{j_i} \frac{q_\beta^{2j_i+1-t}\lambda(K_\beta)-
      q_\beta^{-2j_i+t-1} \lambda(K_\beta)\inv}{q_\beta^t -
      q_\beta^{-t}} m_i.
  \end{align*}
  This is zero if and only if $\lambda(K_\beta) = \pm
  q_\beta^{-2j_i-1+t}$ for some $t = 1,\dots,j_i$. Note that the power
  of $q$ is negative in all cases here so this is not the case by the
  assumption above.  So $F_\beta^{(j_i)}m_i \neq 0$ and we are done
  proving the claim.  So we have $0\neq v_i \in Dm_i$ of weight
  $\lambda$ for $i\in\setN$.
  
  Consider the $U_{q_\beta}(\mathfrak{sl}_2)$ element $C_\beta =
  F_\beta E_\beta + \frac{q_\beta K_\beta + q_\beta\inv
    K_\beta\inv}{(q_\beta- q_\beta\inv)^2}$. Then $C_\beta$ acts on $D
  m_i$ by the scalar
  \begin{equation*}
    \frac{q_\beta \mu_i(K_\beta) + q_\beta\inv \mu_i(K_\beta)\inv}{(q_\beta - q_\beta\inv)}.
  \end{equation*}
  
  If $C_\beta$ acts in the same way on $Dm_i$ and $Dm_k$ then we must
  have either $\mu_i(K_\beta) = \mu_k(K_\beta)$ (i.e. $i=j$) or
  $\mu_i(K_\beta) = q_\beta^{-2} \mu_j(K_\beta)\inv$. The second case
  implies that $\lambda(K_\beta)=\pm q_\beta^{-a}$ for some $a\in
  \setN$ which we have ruled out above. So the vectors $v_i$ are
  linearly independent.  Hence $M$ contains an infinite set of
  linearly independent vectors of weight $\lambda$. This contradicts
  the fact that $M\in \mathcal{F}$.
\end{proof}

\begin{prop}
  \label{prop:2}
  Let $\beta$ be a positive root and $E_\beta$ a root vector
  corresponding to $\beta$. Let $M\in\mathcal{F}$. The set
  $M^{[E_\beta]}=\{m\in M| \dim \left<E_\beta\right> m < \infty \}$ is
  a $U_q$-submodule of $M$.
\end{prop}
\begin{proof}
  Assume first that $\beta$ is a simple root. We want to show that for
  $v\in M^{[E_\beta]}$ we have for each $u\in U_q$, $uv\in
  M^{[E_\beta]}$. It is enough to show this for $u=F_\alpha$,
  $u=K_\alpha$ and $u=E_\alpha$ for all simple roots $\alpha$. If
  $u=K_\alpha$ there is nothing to show since $K_\alpha$ acts
  diagonally on $M$. If $u=F_\alpha$ for $\alpha\neq \beta$ there is
  nothing to show since $E_\beta$ and $F_\alpha$ commute. If
  $\alpha=\beta$ then we get the result from the identity
  \begin{equation*}
    E_\alpha^{(r)}F_\alpha = F_\alpha E_\alpha^{(r)} + E_{\alpha}^{(r-1)} [K_\alpha; r-1]
  \end{equation*}
  found in e.g.~\cite[section 4.4]{Jantzen}. Finally if $u= E_\alpha$
  and $\alpha\neq \beta$ then from the rank $2$ calculations
  in~\cite[section 5.3]{MR1066560} we get:
  \begin{itemize}
  \item If $(\alpha|\beta)=0$:
    \begin{equation*}
      E_\beta^{(r)}E_\alpha = E_\alpha E_\beta^{(r)}.
    \end{equation*}
  \item If $(\alpha|\beta) = -1$:
    \begin{equation*}
      E_\beta^{(r)} E_\alpha = q^r E_\alpha E_\beta^{(r)} + q E_{\alpha+\beta}E_\beta^{(r-1)} 
    \end{equation*}
    where $E_{\alpha+\beta} := T_{s_\alpha}(E_\beta)$.
  \item If $(\alpha|\beta) = -2$ and $\left< \alpha, \beta^\vee\right>
    = -2$:
    \begin{equation*}
      E_\beta^{(r)} E_\alpha = q^{2r} E_\alpha E_\beta^{(r)} + q^{r+1} E_{\alpha+\beta}E_\beta^{(r-1)} + q^2 E_{2\beta + \alpha} E_\beta^{(r-2)}
    \end{equation*}
    where $E_{\alpha+\beta} := T_{s_\alpha}(E_\beta)$ and
    $E_{2\beta+\alpha} := T_{s_\alpha}T_{s_\beta}(E_\alpha)$.
  \item If $(\alpha|\beta) = -2$ and $\left< \alpha, \beta^\vee\right>
    = -1$: In this case we get from the calculations in~\cite[section
    5.3]{MR1066560} that
    \begin{equation*}
      E_\alpha E_\beta^{(r)} = q^{2r} E_\beta^{(r)}E_\alpha + q^2 E_\beta^{(r-1)}E_{\alpha+\beta}
    \end{equation*}
    where $E_{\alpha+\beta} := T_{s_\beta}(E_\alpha)$.
    
    After using the $\setQ$-algebra anti automorphism $\Psi$
    from~\cite[section 1.1]{MR1066560} we get
    \begin{equation*}
      E_\beta^{(r)} E_\alpha = q^{2r} E_\alpha E_\beta^{(r)} + q^2 E'_{\alpha+\beta} E_\beta^{(r-1)}
    \end{equation*}
    where $E'_{\alpha+\beta} = \Psi(E_{\alpha+\beta}) =
    T_{s_\beta}\inv(E_\alpha)$.
  \end{itemize}
  In all cases we get that if $E_\beta^{(n)}m=0$ for $n>>0$ then
  $E_\beta^{(n)}E_\alpha m=0$ for $n>>0$.  This proves that $uv \in
  \{m\in M| \dim \left<E_\beta\right> m < \infty \}$ in this case
  also.

  If $\beta$ is not simple then $E_\beta = T_w(E_{\alpha'})$ for some
  simple root $\alpha'$ and some $w\in W$. Since $T_w$ is an
  automorphism we have $T_w(U_q)=U_q$ so instead of proving the claim
  for $u=E_\alpha$, $K_\alpha$ and $F_\alpha$ we can show it for
  $u=T_w(E_\alpha)$, $T_w(K_\alpha)$ and $T_w(F_\alpha)$ so the claim
  follows from the calculations above.
\end{proof}

\begin{lemma}
  \label{lemma:5}
  Let $E_\beta$ and $E_\beta'$ be two choices of root vectors. Then
  $M^{[E_\beta]}=M^{[E_\beta']}$
\end{lemma}
\begin{proof}
  Suppose we have two root vectors $E_{\beta}$ and $E_\beta'$. By
  Proposition~\ref{prop:2} and Proposition~\ref{prop:1} we have $\dim
  \left<E_\beta'\right> m < \infty$ for all $m\in M^{[E_\beta]}$ so
  $M^{[E_\beta]}\subset M^{[E_\beta']}$. Symmetrically we have also
  $M^{[E_\beta']}\subset M^{[E_\beta]}$.
\end{proof}

\begin{defn}
  Let $\beta$ be a positive root and $E_\beta$ a root vector
  corresponding to $\beta$.  Define $M^{[\beta]} = \{m\in M| \dim
  \left<E_\beta\right> m < \infty \}$.
\end{defn}
By Lemma~\ref{lemma:5} this definition is independent of the chosen
root vector.

Everything here that is done for a positive root $\beta$ can be done
for a negative root just by replacing the $E$'s with $F$'s, i.e. for a
negative root $\beta\in \Phi^-$, $M^{[\beta]} = \{m\in M| \dim \left<
  F_{-\beta} \right> m < \infty \}$ and so on.

\begin{defn}
  Let $M\in \mathcal{F}$. Let $\beta\in\Phi$. $M$ is called
  $\beta$-free if $M^{[\beta]}=0$.
\end{defn}

Note that $M$ is $\beta$-finite if and only if $M^{[\beta]}=M$ so
$\beta$-free is, in a way, the opposite of being
$\beta$-finite. Suppose $L\in \mathcal{F}$ is a simple module and
$\beta$ a root. Then by Proposition~\ref{prop:2} $L$ is either
$\beta$-finite or $\beta$-free.

\begin{defn}
  Let $M\in \mathcal{F}$. Define $F_M = \{\beta \in \Phi| \text{$M$ is
    $\beta$-finite}\}$ and $T_M = \{ \beta \in \Phi | \text{$M$ is
    $\beta$-free} \}$. For later use we also define $F_M^s := F_M \cap
  (-F_M)$ and $T_M^s := T_M \cap (-T_M)$ to be the symmetrical parts
  of $F_M$ and $T_M$.
\end{defn}
Note that $\Phi = F_L \cup T_L$ for a simple module $L$ and this is a
disjoint union.

\begin{defn}
  \label{def:torsion-free}
  A module $M$ is called torsion free if $T_M = \Phi$.
\end{defn}

\begin{prop}
  \label{prop:3}
  Let $L$ be a simple module and $\beta$ a root. $L$ is $\beta$-free
  if and only if $q^{\setN \beta}\wt L \subset \wt L$.
\end{prop}
\begin{proof}
  Assume $L$ is $\beta$-free and $\beta\in \Phi^+$. Let $E_{\beta}$ be
  a corresponding root vector. The proof is similar for $\beta\in
  \Phi^-$ but with $F$ instead of $E$. Then for all $0 \neq m\in L$,
  $E_\beta^{(r)}m \neq 0$. If $\lambda \in \wt L$ then there exists $0
  \neq m_\lambda \in L_\lambda$ and since $E_\beta^{(r)}m_\lambda \in
  L_{q^{r\beta}\lambda}$ the implication follows. For the other way
  assume $q^{\setN \beta}\wt L \subset \wt L$. Then $L$ is clearly not
  $\beta$-finite. Since $L$ is simple $L$ must then be $\beta$-free.
\end{proof}

\begin{prop}
  \label{prop:4}
  Let $L\in\mathcal{F}$ be a simple module. $T_L$ is a closed subset
  of the roots $\Phi$. That is if $\beta,\gamma \in T_L$ and $\beta +
  \gamma \in \Phi$. Then $\beta + \gamma \in T_L$.
\end{prop}
\begin{proof}
  Since $L$ is $\beta$-free we have $q^{\setN \beta}\wt L \subset \wt
  L$ and since $L$ is $\gamma$ free we get further $q^{\setN
    \gamma}q^{\setN \beta} \wt L \subset \wt L$ so therefore
  $q^{\setN(\beta+\gamma)}\wt L \subset \wt L$ hence $L$ is $(\beta +
  \gamma)$ free.
\end{proof}

\begin{prop}
  \label{prop:5}
  Let $M\in \mathcal{F}$ be a $U_q$-module. $F_M$ is a closed subset
  of $\Phi$. That is if $\beta,\gamma \in F_M$ and $\beta + \gamma \in
  \Phi$ then $\beta + \gamma \in F_M$.
\end{prop}
\begin{proof}
  Let $\alpha,\beta\in F_M$ with $\alpha+\beta\in\Phi$. We have to
  show that $\alpha+\beta\in F_M$. First let us show the claim if the
  root system $\Phi$ is a rank $2$ root system. In this case the claim
  will follow from the rank $2$ calculations
  in~\cite{MR1066560}. Assume $\Pi = \{\alpha_1,\alpha_2\}$. Assume
  first that we have $\alpha\in\Pi$ and $\beta\in \Phi^+$. We show
  below that we can always reduce to this situation. We can assume
  $\alpha=\alpha_1$ by renumbering if neccesary. We now have 5
  possibilites:

  Case 0) $(\alpha_1,\alpha_2)=0$ is clear.

  Case 1): $(\alpha_1|\alpha_2)=-1$. The only possibility for
  $\beta\in \Phi^+$ such that $\alpha+\beta$ is a root is
  $\beta=\alpha_2$. Set $E_{\alpha+\beta} = T_{s_\beta}(E_{\alpha})$
  then Lusztig shows in~\cite[section 5.5]{MR1066560} that
  \begin{equation*}
    E_{\alpha+\beta}^{(k)} = \sum_{t=0}^k (-1)^{t}q^{-t} E_\beta^{(k-t)}E_\alpha^{(k)}E_\beta^{(t)}.
  \end{equation*}
  The difference in the definition of the braid operators
  between~\cite{Jantzen} and~\cite{MR1066560} means that we have to
  multiply the formula in~\cite{MR1066560} by $(-1)^{k}$ since (using
  the notation of~\cite{MR1066560}) $E_{12} = -E_{\alpha+\beta}$.  Let
  $m\in M$. Then there exists a $T\in \setN$ such that
  $E_\beta^{(t)}m=0$ for $t\geq T$ since $M$ is $\beta$-finite. Let
  $m_t = E_\beta^{(t)}m$, $t=0,1,\dots,T$. For each $m_t$ there is a
  $K_t\in \setN$ such that $E_\alpha^{(k)}m_t = 0$ for $k\geq K_t$
  since $M$ is $\alpha$-finite. Set $K=\max\{T,K_0,\dots,K_T\}$ then
  the above identity shows that $E_{\alpha+\beta}^{(k)}m = 0$ for
  $k\geq K$

  Case 2): $\left<\alpha_1,\alpha_2^\vee\right> = -2$. In this case
  $\beta = \alpha_2$ is the only possibility to choose
  $\beta\in\Phi^+$ such that $\alpha+\beta\in\Phi$. Set
  $E_{\alpha+\beta} = T_\alpha(E_\beta)$ then by~\cite[section
  5.5]{MR1066560}:
  \begin{equation*}
    E_{\alpha+\beta}^{(k)} = \sum_{t=0}^k (-1)^{t}q^{-2t} E_\alpha^{(k-t)}E_\beta^{(k)}E_\alpha^{(t)}
  \end{equation*}
  and the same argument as above works.
  
  Case 3): $\left<\alpha_2,\alpha_1^\vee\right> = -2$ and $\beta=
  \alpha_2$. Set $E_{\alpha+\beta} = T_{\beta}(E_\alpha)$ then
  \begin{equation*}
    E_{\alpha+\beta}^{(k)} = \sum_{t=0}^k (-1)^{t}q^{-2t} E_\beta^{(k-t)}E_\alpha^{(k)}E_\beta^{(t)}
  \end{equation*}
  and the argument follows like in case 1) and 2).

  Case 4): $\left<\alpha_2,\alpha_1^\vee\right> = -2$ and $\beta=
  \alpha_1 + \alpha_2$. In this case set $E_\beta =
  E_{\alpha_1+\alpha_2}= T_{\alpha_2}(E_{\alpha_1})$ and
  $E_{\alpha+\beta} = E_{2\alpha_1+\alpha_2}
  =T_{\alpha_2}T_{\alpha_1}(E_{\alpha_2})$. We want a property similar
  to the one in the other cases. We want to show that there exists
  $c_t\in \setQ(q)$ such that
  \begin{equation*}
    E_{2\alpha_1+\alpha_2}^{(k)} = \sum_{t=0}^k c_t E_{\alpha_1}^{(k-t)}E_{\alpha_1+\alpha_2}^{(k)} E_{\alpha_1}^{(t)}.
  \end{equation*}
  We will use notation like in~\cite{MR1066560} so set
  $E_1=E_{\alpha_1}$, $E_{12}=E_{\alpha_1+\alpha_2}$ and $E_{112}=
  E_{2\alpha_1+\alpha_2}$. Let $k\in \setN$. By 5.3 (h)
  in~\cite{MR1066560}
  \begin{equation*}
    E_1^{(k)}E_{12}^{(k)} = (-1)^{k}q^k \prod_{i=1}^k (q^{2i}+1) E_{112}^{(k)} + \sum_{s=0}^{k-1} (-1)^s q^{s-s(k-s)-s(t-s)}\left( \prod_{i=1}^s (q^{2i}+1) \right) E_{12}^{(k-s)}E_{112}^{(s)} E_{1}^{(k-s)}
  \end{equation*}
  so
  \begin{equation*}
    E_{112}^{(k)} = (-1)^k c
    \left( E_1^{(k)}E_{12}^{(k)} - \sum_{s=0}^{k-1} (-1)^s q^{s-s(k-s)-s(t-s)}\left( \prod_{i=1}^s (q^{2i}+1) \right) E_{12}^{(k-s)}E_{112}^{(s)} E_{1}^{(k-s)}\right)
  \end{equation*}
  where $c=\left(q^k \prod_{i=1}^k (q^{2i}+1)\right)\inv$.
  
  We will show by induction over $s<k$ that there exists $a_i\in
  \setQ(q)$ such that
  \begin{equation*}
    E_{12}^{(k-s)}E_{112}^{(s)}E_1^{(k-s)} = \sum_{i=0}^s a_i E_1^{(i)} E_{12}^{(k)} E_{1}^{(k-i)}.
  \end{equation*}
  The induction start $s=0$ is obvious. Now observe that again from
  5.3 (h) in~\cite{MR1066560} we have for $s<k$:
  \begin{align*}
    E_1^{(s)} E_{12}^{(k)} =& (-1)^s q^{s-s(k-s)}\prod_{i=1}^s
    (q^{2i}+1) E_{12}^{(k-s)} E_{112}^{(s)}
    \\
    &+ \sum_{n=0}^{s-1} (-1)^n q^{n-n(s-n)-n(k-n)}\left( \prod_{i=1}^n
      (q^{2i}+1) \right) E_{12}^{(k-n)} E_{112}^{(n)} E_{1}^{(s-n)}.
  \end{align*}
  So
  \begin{equation*}
    E_{12}^{(k-s)} E_{112}^{(s)} = (-1)^s \left(q^{s-s(k-s)}\prod_{i=1}^s (q^{2i}+1)\right)\inv \left( E_1^{(s)} E_{12}^{(k)} - \sum_{n=0}^{s-1} (-1)^n b_n E_{12}^{(k-n)} E_{112}^{(n)} E_{1}^{(s-n)}\right)
  \end{equation*}
  where $b_n\in \setQ(q)$ are the coefficients above. Hence
  \begin{equation*}
    E_{12}^{(k-s)} E_{112}^{(s)} E_1^{(k-s)} = (-1)^s b E_1^{(s)} E_{12}^{(k)}E_1^{(k-s)} + \sum_{n=0}^{s-1} (-1)^{s+n}b_n' E_{12}^{(k-n)} E_{112}^{(n)} E_{1}^{(k-n)}
  \end{equation*}
  for some coefficients $b$ and $b_n' \in\setQ(q)$. This identity
  completes the induction over $s$.
  
  So to sum up we have proven that there exists $c_t\in \setQ(q)$ such
  that
  \begin{equation*}
    E_{2\alpha_1+\alpha_2}^{(k)} = \sum_{t=0}^k c_t E_{\alpha_1}^{(k-t)}E_{\alpha_1+\alpha_2}^{(k)} E_{\alpha_1}^{(t)}.
  \end{equation*}
  (Note for later use in the root of unity case that the $c_t$ are in
  the localization of $\setZ[q,q\inv]$ in the elements $(q^{2i}+1)$
  for $i\in\setN$ which are nonzero unless $q$ is an $l$th root of
  unity with $l$ even).  Now the proof goes as above.

  The above 5 cases are the only possible cases with the above
  assumptions since we have excluded $G_2$.

  We will now show how to reduce the problem to rank $2$.  Assume
  $\beta,\gamma\in F_M$ and $\beta+\gamma \in \Phi$. We will first
  show:
  
  \begin{itemize}
  \item There exists a $w\in W$ such that $w(\beta)\in\Pi$ and
    $w(\gamma)\in \Phi^+$.
  \end{itemize}
  
  Let $w_0=s_{i_1}\cdots s_{i_N}$ be a reduced expression and let
  $\beta_j = s_{i_1}\cdots s_{i_{j-1}}(\alpha_{i_j})$. Then $\Phi^+ =
  \{\beta_1,\dots,\beta_N\}$. Assume first that both $\beta$ and
  $\gamma$ are positive. Then $\beta=\beta_j$ and $\gamma=\beta_r$ for
  some $j$ and $r$. Without loss of generality we can assume
  $j<r$. Then we can set $w=s_{i_{j-1}}\cdots s_{i_{1}}$ in this
  case. If $\beta$ and $\gamma$ are both negative then $w_0(\beta)$
  and $w_0(\gamma)$ are both positive and we can do as before. Assume
  $\beta<0$ and $\gamma>0$. Assume $\beta=-\beta_j$ and $\gamma=
  \beta_r$ for some $j$ and $r$. Without loss of generality we can
  assume $j<r$. Then set $w=s_{i_{j}}\cdots s_{i_1}$. The claim has
  been shown.

  Next we will show:
  \begin{itemize}
  \item There exists a $w\in W$ such that $w(\beta)$ and $w(\gamma)$
    is contained in a rank $2$ subsystem of the roots.
  \end{itemize}
  If $(\beta|\gamma)<0$ then there exists a simple system $\Pi'$ of
  $\Phi$ such that $\beta$ and $\gamma$ are in $\Pi'$. But since all
  simple system of a root system are $W$ conjugate the claim
  follows. Assume $(\beta|\gamma)\geq 0$. Then
  $\left<\beta+\gamma,\gamma^\vee\right> \geq
  \left<\gamma|\gamma^\vee\right> = 2$ so $s_\gamma( \beta + \gamma) =
  \beta + \gamma - \left<\beta+\gamma,\gamma^\vee\right>\gamma \leq
  \beta - \gamma$. So $\beta-\gamma$ is a root in this case. Since we
  have excluded $G_2$ this means that the $\gamma$ string through
  $\beta$ is $\beta-\gamma,\beta,\beta+\gamma$ and therefore
  $\left<\beta+\gamma,\gamma^\vee\right> = 2$ or equivalently $\left<
    \beta,\gamma^\vee\right> = 0$. So
  $(\beta-\gamma|\gamma)=-(\gamma|\gamma)<0$. Hence there is a simple
  system of roots $\Pi'$ such that $\gamma,\beta-\gamma\in \Pi'$. So
  there exists $w$ such that $w(\gamma)$ and $w(\beta-\gamma)$ are
  simple roots. Since $w(\beta) = w(\gamma) + w(\beta-\gamma)$ we see
  that $w(\beta)$ and $w(\gamma)$ are contained in a rank $2$
  subsystem of $\Phi$.  So the second claim is proven.
  
  Note that ${^w}M$ is $w(\beta)$ and $w(\gamma)$ finite: Since $\wt
  {^w}M = w(\wt M)$ we have that a $\mu \in \wt {^w}M$ is of the form
  $\mu = w(\lambda)$ for some $\lambda \in \wt M$. Now $q^{\setN
    w(\beta)}\mu \cap \wt {^w}M = w( q^{\setN \beta}\lambda \cap \wt
  M)$ is finite because $M$ was $\beta$-finite. All in all we get that
  for some $w$ we have $w(\beta+\gamma)\in F_{{^w}M}$. But since
  $F_{{^w}M}=w(F_M)$ this shows that $\beta+\gamma\in F_M$.
\end{proof}

Let $L$ be a simple module. Since $F_L$ and $T_L$ are both closed
subsets of $\Phi$ we get from~\cite[Lemma 4.16]{Fernando} that $P_L:=
F_L \cup T_L^s$ is a parabolic subset of the roots - i.e. $P_L \cup
(-P_L) = \Phi$ and $P_L$ is a closed subset of $\Phi$.

Since $P_L \cup (-P_L) = \Phi$ we must have for some $w\in W$, $\Phi^+
\subset w(P_L)$. From now on we will assume $\Phi^+ \subset P_L$ since
otherwise we can just describe the module ${^w}L$ and then untwist
once we have described this module.  So we assume $P_L = \Phi^+ \cup
\left<\Pi'\right>$ where $\Pi' \subset \Pi$ and where
$\left<\Pi'\right>$ denotes the subset of $\Phi$ generated by $\Pi'$,
i.e. $\left<\Pi'\right>=\setZ\Pi' \cap \Phi$.

Let $\mathfrak{p}$ be the parabolic Lie algebra corresponding to $P_L$
i.e. $\mathfrak{p} = \mathfrak{h}\oplus \bigoplus_{\beta\in P_L}
\mathfrak{g}_\beta$ and let $\mathfrak{l}$ and $\mathfrak{u}$ be the
Levi part and the nilpotent part of $\mathfrak{p}$ respectively
i.e. $\mathfrak{l} = \mathfrak{h}\oplus \bigoplus_{\beta\in P_L^s}
\mathfrak{g}_\beta$ and $\mathfrak{u} = \bigoplus_{\beta\in
  P_L\backslash P_L^s} \mathfrak{g}_\beta$. We can define
$U_q(\mathfrak{p})$, $U_q(\mathfrak{l})$ and
$U_q(\mathfrak{u})$. Furthermore we can define $U_q(\mathfrak{u}^-)$
where $\mathfrak{u^-}$ is the nilpotent part of the opposite parabolic
$\mathfrak{p}^-$ corresponding to $(-P_L)$.  We have
$U_q(\mathfrak{p}) = U_q(\mathfrak{l}) U_q(\mathfrak{u})$ and
$U_q(\mathfrak{g}) = U_q(\mathfrak{u^-}) U_q(\mathfrak{p})$.

Here is how we define the above subalgebras: (Defined like
in~\cite{deg_of_parabolic}) Assume $P_L = \Phi^+ \cup \left< \Pi'
\right>$. Let $w_0^{\mathfrak{l}}$ be the longest element in the Weyl
group $W^{\mathfrak{l}}$ corresponding to $\Pi'$. Let $w_0$ be the
longest element in $W$. Set
$\bar{w}=w_0(w_0^{\mathfrak{l}})\inv$. Choose a reduced expression
$w_0=s_{j_1}\cdots s_{j_k}s_{i_1}\cdots s_{i_h}$ such that
$w_0^\mathfrak{l} = s_{i_1}\cdots s_{i_h}$. Let
$\{E_\beta,F_\beta|\beta\in\Phi^+\}$ be the root vectors defined by
this reduced expression.

Set
\begin{align*}
  \beta_t^1 &= \beta_{t+k} = {\bar{w}}{s_{i_1}}\cdots
  {s_{i_{t-1}}}(\alpha_{i_t}), \quad t = 1,\dots,h
  \\
  \beta_t^2 &= \beta_t = {s_{j_1}}\cdots {s_{j_{t-1}}}(\alpha_{j_t}), \quad
  t = 1,\dots, k.
\end{align*}
This means that
\begin{align*}
  F_{\beta_t^1} &= T_{\bar{w}}T_{s_{i_1}}\cdot
  T_{s_{i_{t-1}}}(F_{\alpha_{i_t}}), \quad t = 1,\dots,h
  \\
  F_{\beta_t^2} &= T_{s_{j_1}}\cdot T_{s_{j_{t-1}}}(F_{\alpha_{j_t}}), \quad
  t = 1,\dots, k
\end{align*}
and similarly for the $E$'s.

We define
\begin{equation*}
  U_q(\mathfrak{p}) = \left< E_{\beta_j} , K_\mu, F_{\beta_{i}^1} \right>_{j=1,\dots,N,\mu\in Q,i=1,\dots h},
\end{equation*}
\begin{equation*}
  U_q(\mathfrak{l}) = \left< E_{\beta_i^1} , K_\mu, F_{\beta_{i}^1} \right>_{\mu\in Q,i=1,\dots h}
\end{equation*}
and
\begin{equation*}
  U_q(\mathfrak{u}) = \left< E_{\beta_i^2} \right>_{i=1,\dots,k}.
\end{equation*}

Similarly we define $U_q(\mathfrak{u}^-) = \left< F_{\beta_i^2}
\right>_{i=1,\dots,h}$. All of these are subalgebras of
$U_q(\mathfrak{g})$ are independent of the chosen reduced expression
of $w_0$ and $w_0^{\mathfrak{l}}$. Furthermore $U_q(\mathfrak{p})$ and
$U_q(\mathfrak{l})$ are Hopf subalgebras of $U_q(\mathfrak{g})$ as
stated in~\cite[Proposition~5 and Lemma~2]{deg_of_parabolic}.

There is a $Q$ grading on $U_q$ with $\deg E_\alpha=\alpha$, $\deg
F_\alpha = -\alpha$ and $\deg K_\beta^{\pm 1} = 0$ as described in
e.g.~\cite[section 4.7]{Jantzen}. This induces a grading on
$U_q^{\pm}$ and on $U_q(\mathfrak{u})$ and $U_q(\mathfrak{u}^-)$.  We
will define $U_q(\mathfrak{u})^{>0}$ and $U_q(\mathfrak{u}^-)^{<0}$ to
be the subalgebras consisting of elements with nonzero degree
(i.e. the augmentation ideals).

\begin{defn}
  Let $\mathfrak{p}$ be a standard parabolic sub Lie algebra of
  $\mathfrak{g}$ and let $\mathfrak{l}$, $\mathfrak{u}$ and
  $\mathfrak{u^-}$ be defined as above. Let $N$ be a
  $U_q(\mathfrak{l})$-module. We define
  \begin{equation*}
    \mathcal{M}(N) = U_q(\mathfrak{g}) \tensor_{U_q(\mathfrak{p})} N,
  \end{equation*}
  where $N$ is considered as a $U_q(\mathfrak{p})$-module with
  $U_q(\mathfrak{u})$ acting trivially, i.e. through the coidentity
  $\epsilon: U_q(\mathfrak{u})\to \setC$ sending everything of nonzero
  degree to zero.
\end{defn}

\begin{defn}
  If $M$ is a $U_q(\mathfrak{g})$-module we define
  \begin{equation*}
    M^{\mathfrak{u}} = \{ m\in M | x m = \epsilon(x)m, \, x\in U_q(\mathfrak{u}) \}.
  \end{equation*}
\end{defn}

\begin{prop}
  \label{prop:6}
  Let $M$ be a $U_q(\mathfrak{g})$-module. $M^{\mathfrak{u}}$ is a
  $U_q(\mathfrak{l})$-module.
\end{prop}
\begin{proof}
  We will show that for $u\in U_q(\mathfrak{l})$,
  $U_q(\mathfrak{u})^{>0} u \cap U_q(\mathfrak{g})
  U_q(\mathfrak{u})^{>0} \neq \emptyset$. This is true by simple
  grading considerations. We know that $U_q(\mathfrak{u})^{>0}u
  \subset U_q(\mathfrak{l}) U_q(\mathfrak{u}) =
  U_q(\mathfrak{l})U_q(\mathfrak{u})^{>0}+U_q(\mathfrak{l})$. But the
  degree of a homogeneous element $u'u\in U_q^-$ with $u'\in
  U_q(\mathfrak{u})^{>0}$ cannot be in $\setZ \Pi'$ since that would
  mean $u' \in U_q(\mathfrak{l})$. So $U_q(\mathfrak{u})^{>0}u \subset
  U_q(\mathfrak{l})U_q(\mathfrak{u})^{>0}$.
\end{proof}

\begin{prop}
  \label{prop:7}
  Let $N$ be a $U_q(\mathfrak{l})$-module and let $M$ be a
  $U_q(\mathfrak{g})$-module. There are natural vector space
  isomorphisms
  \begin{equation*}
    \Phi = \Phi_{M,N}: \Hom_{U_q(\mathfrak{g})}(\mathcal{M}(N),M) \iso \Hom_{U_q(\mathfrak{l})}(N,M^{\mathfrak{u}}).
  \end{equation*}
\end{prop}
\begin{proof}
  If $f:\mathcal{M}(N)\to M$ is a $U_q(\mathfrak{g})$-module map then
  $\Phi(f):N\to M^{\mathfrak{u}}$ is defined by
  $\Phi(f)=f^{\mathfrak{u}}\circ (1\tensor \id_N)$, where $1 \tensor
  \id_N : N \to \mathcal{M}(N)^{\mathfrak{u}}$ is given by $n\mapsto
  1\tensor n$ and $f^{\mathfrak{u}}:\mathcal{M}(N)^{\mathfrak{u}}\to
  M^{\mathfrak{u}}$ is the restriction of $f$ to
  $\mathcal{M}(N)^{\mathfrak{u}}$.
  
  The inverse map $\Psi$ is given by: For $g:N\to M^{\mathfrak{u}}$,
  $\Psi(g)(u\tensor n) = u g(n)$. It is easy to check that $\Phi$ and
  $\Psi$ are inverse to each other.
\end{proof}

\begin{prop}
  \label{prop:8}
  If $X$ is a simple $U_q(\mathfrak{l})$-module then $\mathcal{M}(X)$
  has a unique simple quotient $L(X)$.
\end{prop}
\begin{proof}
  The proof is exactly the same as the proof of Proposition 3.3
  in~\cite{Fernando}: Suppose $M$ is a submodule of
  $\mathcal{M}(X)$. If $0\neq v\in M\cap (1\tensor X)$ then $U_q v =
  U_q U_q(\mathfrak{l})v = U_q (1\tensor X) = \mathcal{M}(X)$ so
  $M\cap (1\tensor X)=0$ for every proper submodule $M$. Let $N$ be
  the sum of all proper submodules. $N$ is proper since $N\cap
  (1\tensor X)=0$ and maximal since it is the sum of all proper
  submodules.
\end{proof}

Let $\mathcal{F}(\mathfrak{l})$ denote the full subcategory of
$U_q(\mathfrak{l})$-modules that consists of modules that are finitely
generated over $U_q$ and are weight modules with finite dimensional
weight spaces.

\begin{prop}
  \label{prop:9}
  The maps $L:N \mapsto L(N)$ and $F:V\mapsto V^{\mathfrak{u}}$
  determine a bijective correspondence between the simple modules in
  $\mathcal{F}(\mathfrak{l})$ and the simple modules $M$ in
  $\mathcal{F}(\mathfrak{g})$ that have $M^{\mathfrak{u}}\neq 0$. $L$
  and $F$ are inverse to each other.
\end{prop}
The second part of the proof is just a quantum version of the proof of
Proposition~3.8 in~\cite{Fernando}. The first part is shown a little
differently here.
\begin{proof}
  First we will show that if $V$ is a simple
  $U_q(\mathfrak{g})$-module with $V^{\mathfrak{u}}\neq 0$ then
  $V^{\mathfrak{u}}$ is a simple $U_q(\mathfrak{l})$-module: Assume
  $0\neq V_1 \subset V^{\mathfrak{u}}$ is a
  $U_q(\mathfrak{l})$-submodule of $V^{\mathfrak{u}}$. We will show
  that $V_1=V^{\mathfrak{u}}$. Since $V$ is a simple
  $U_q(\mathfrak{g})$-module we have $V=U_q(\mathfrak{g})V_1$. Now as
  a vector space we have
  \begin{align*}
    V = U_q(\mathfrak{g})V_1 = U_q(\mathfrak{u^-}) U_q(\mathfrak{l})
    U_q(\mathfrak{u}) V_1 &= U_q(\mathfrak{u^-}) U_q(\mathfrak{l})
    (U_q(\mathfrak{u})^{>0}+\setC) V_1
    \\
    &= U_q(\mathfrak{u^-}) U_q(\mathfrak{l}) V_1
    \\
    &= U_q(\mathfrak{u^-}) V_1
    \\
    &= (U_q(\mathfrak{u^-})^{<0} + \setC)V_1
    \\
    &= U_q(\mathfrak{u}^-)^{<0}V_1 + V_1.
  \end{align*}
  
  We are done if we show $U_q(\mathfrak{u}^-)^{<0}V^{\mathfrak{u}}\cap
  V^{\mathfrak{u}}=0$. Observe that
  $U_q(\mathfrak{u}^-)^{<0}V^{\mathfrak{u}}$ is a $U_q(\mathfrak{l})$
  module since
  $U_q(\mathfrak{l})U_q(\mathfrak{u}^-)^{<0}=U_q(\mathfrak{u}^-)^{<0}U_q(\mathfrak{l})$. Assume
  $v\in V^{\mathfrak{u}}$ and assume we have a $u'\in
  U_q(\mathfrak{u^-})^{<0}$ such that $u'v \in V^{\mathfrak{u}}$. We
  can assume $u'\in (U_q(\mathfrak{u}^-)^{<0})_{\gamma}$ for some
  $\gamma\in Q$ and $v\in V_\mu$ for some $\mu\in X$. Assume $u'v\neq
  0$. Then since $V$ is simple there exists a $u\in U_q$ such that
  $uu' v = v$ but by weight considerations we must have $u\in
  (U_q)_{-\gamma}\subset U_q(\mathfrak{p}^-)U_q(\mathfrak{u})^{>0}$ so
  $u u'v = 0$ since $u'v\in V^\mathfrak{u}$. A contradiction.

  Now assume $N$ is a simple $U_q(\mathfrak{l})$
  module. $L(N)^{\mathfrak{u}}$ is simple by the above. Let $\Phi$ be
  the isomorphism from Proposition~\ref{prop:7} and consider
  $\Phi(p):N \to L(N)^{\mathfrak{u}}$ where $p:\mathcal{M}(N)\to L(N)$
  is the cannocial projection from $\mathcal{M}(N)$ to $L(N)$. Since
  $\Phi$ is an isomorphism the map $\Phi(p)$ is nonzero. Since $N$ is
  simple by assumption and $L(N)^{\mathfrak{u}}$ is simple by the
  above we get that $\Phi(p)$ is an isomorphism.

  Suppose $V$ is a simple $U_q(\mathfrak{g})$-module such that
  $V^{\mathfrak{u}}$ is nonzero. Let
  $f=\Phi\inv(\id):M(V^{\mathfrak{u}})\to V$ where $\id:
  V^{\mathfrak{u}}\to V^{\mathfrak{u}}$ is the identity map. Then $f$
  is nonzero and therefore surjective because $V$ is simple. But since
  $L(V^{\mathfrak{u}})$ is the unique simple quotient of
  $M(V^{\mathfrak{u}})$ we get $L(V^{\mathfrak{u}})=V$.
\end{proof}

Let $\mathfrak{p}$ be a standard parabolic subalgebra of
$\mathfrak{g}$ and define $U_q(\mathfrak{p})$, $\mathfrak{l}$,
$U_q(\mathfrak{l})$ etc. as above. Let $\Phi^{\mathfrak{l}}$ be the
roots corresponding to $\mathfrak{l}$ i.e. such that $\mathfrak{l} =
\mathfrak{h} \oplus \bigoplus_{\beta\in \Phi^{\mathfrak{l}}}
\mathfrak{g}_\beta$. Then for $\beta\in \Phi^{\mathfrak{l}}$ and a
$U_q(\mathfrak{l})$-module $M$ we define $\beta$-finite, $\beta$-free,
$M^{[\beta]}$ etc. as above. The definitions, lemmas and propositions
above still hold in this case as long as we require $\beta \in
\Phi^{\mathfrak{l}}$ so that we actually have root vectors
$E_\beta,F_\beta\in U_q(\mathfrak{l})$. We define $T_M:=\{\beta\in
\Phi^{\mathfrak{l}}| M^{[\beta]}=0\}$ and $F_M:=\{\beta\in
\Phi^{\mathfrak{l}}| M^{[\beta]}=M\}$ i.e. as before but only for
roots in $\Phi^{\mathfrak{l}}$.

By now we have reduced the problem of classifying simple modules in
$\mathcal{F}(\mathfrak{g})$ somewhat. If $L\in \mathcal{F}$ is a
simple module we know that there exists some $w$ such that
$\Phi^+\subset P_{{^w}L}$. Define $\mathfrak{l}$, $U_q(\mathfrak{l})$
from $L$ etc. as above, then $\Phi^{\mathfrak{l}} = \left< \Pi'
\right>=F_L^s \cup T_L^s$ where $\Pi'$ is the subset of simple roots
such that $P_L = \Phi^+ \cup \left< \Pi' \right>$. From the above we
get then that ${^w}L$ is completely determined by the simple
$U_q(\mathfrak{l})$-module $({^w}L)^{\mathfrak{u}}$. So we have
reduced the problem to looking at simple $U_q(\mathfrak{l})$-modules
$N$ satisfying $\Phi^{\mathfrak{l}} =F_N^s \cup T_N^s$.

We claim that $\Pi' = \Pi_{F_N^s}' \cup \Pi_{T_N^s}'$ such that $F_N^s
= \left< \Pi_{F_N}' \right>$ and $T_N^s = \left< \Pi_{T_N}' \right>$
and such that none of the simple roots in $\Pi_{F_N^s}'$ are connected
to any simple root from $\Pi_{T_N^s}'$. Suppose $\alpha \in F_N^s$ is
a simple root and suppose $\alpha'\in \Pi'$ is a simple root that is
connected to $\alpha$ in the Dynkin diagram. So $\alpha+\alpha'$ is a
root. There are two possibilities. Either $\alpha+\alpha' \in F_N$ or
$\alpha+\alpha'\in T_N$. If $\alpha+\alpha'\in F_N$: Since $F_N^s$ is
symmetric we have $-\alpha\in F_N^s$ and since $F_N$ is closed
$\alpha' = \alpha+\alpha' + (-\alpha) \in F_N$. If $\alpha+\alpha'\in
T_N$ and $\alpha' \in T_N$ then we get similarly $\alpha \in T_N$
which is a contradiction. So $\alpha'\in F_N$. We have shown that if
$\alpha\in F_N$ then any simple root connected to $\alpha$ is in $F_N$
also. So $F_N$ and $T_N$ contains different connected components of
the Dynkin diagram for $\Phi^\mathfrak{l}$.

Let $\tau = c(\mathfrak{l}) \oplus \mathfrak{g}_{F_N^s} \oplus
\mathfrak{h}_{F_N^s}$ and $\mathfrak{t} = \mathfrak{g}_{T_N^s}\oplus
\mathfrak{h}_{T_N^s}$. Define
\begin{equation*}
  U_q(\tau) = \left< E_\alpha,K_\alpha,K_\beta,F_\alpha\right>_{\alpha\in \Pi_{F_N^s}',\beta\in \Phi \backslash \Phi^\mathfrak{l}}
\end{equation*}
and
\begin{equation*}
  U_q(\mathfrak{t}) = \left< E_\alpha,K_\alpha,F_\alpha\right>_{\alpha\in \Pi_{T_N^s}'}.
\end{equation*}
Then by construction $U_q(\mathfrak{g}) \iso U_q(\tau) \tensor_\setC
U_q(\mathfrak{t})$ as a vector space via $u_1\tensor u_2 \mapsto
u_1u_2$ for $u_1\in U_q(\tau)$ and $u_2\in U_q(\mathfrak{t})$.

To continue we want to use a result similar to~\cite{Lemire} Theorem 1
which says that there is a 1-1 correspondence between simple
$U_q(\mathfrak{l})$-modules and simple $(U_q(\mathfrak{l}))_0$
modules. Since Lemire's result is for Lie algebras we will prove the
same for quantum group modules but the proofs are essentially the
same. In the following $\mathfrak{l}$ is the Levi part of some
standard parabolic subalgebra $\mathfrak{p}$ and $U_q(\mathfrak{l})$
is defined as above. Note in particular that the results work for
$\mathfrak{l}=\mathfrak{g}$ by choosing
$\mathfrak{p}=\mathfrak{g}$. For easier notation we will set $C_q :=
(U_q(\mathfrak{l}))_0$.

\begin{lemma}
  \label{lemma:1}
  Let $V$ be a simple $U_q(\mathfrak{l})$-module and $\lambda$ a
  weight of $V$. Then $V_\lambda$ is a simple $C_q$-module.
\end{lemma}
\begin{proof}
  It is enough to show that for $v\in V_\lambda$ nonzero we have
  $V_\lambda = C_q v$ but this follows since $V_\lambda =
  (U_q(\mathfrak{l}) v)_\lambda = (\bigoplus_\nu U_q(\mathfrak{l})_\nu
  v)_\lambda = U_q(\mathfrak{l})_0 v$
\end{proof}

\begin{lemma}
  \label{lemma:2}
  Assume $V_1$ and $V_2$ are simple $U_q(\mathfrak{l})$-modules. Let
  $\lambda\in \wt V_1$ and assume $(V_1)_\lambda \iso (V_2)_\lambda$
  as $C_q$-modules. Then $V_1 \iso V_2$.
\end{lemma}
\begin{proof}
  Let $0\neq v_i \in (V_i)_\lambda$, $i=1,2$. Then $(V_i)_\lambda \iso
  C_q/ \Ann_{C_q}(v_i)$ as $C_q$-modules since $(V_i)_\lambda$ is
  simple (Lemma~\ref{lemma:1}). Let $M=\Ann_{C_q}(v_1)$, then $M$ is a
  maximal left ideal in $C_q$ since $C_q / M$ is simple. We will show
  that there exists a unique maximal ideal $M'$ of $U_q(\mathfrak{l})$
  containing $M$. Let $M'' = U_q(\mathfrak{l})M$. Then $M'' \neq
  U_q(\mathfrak{l})$ because $M \neq C_q$ and so there is a maximal
  ideal $M'$ containing $M''$. To show uniqueness we will show that
  $U_q(\mathfrak{l}) / M''$ has a unique maximal submodule (and
  therefore a unique simple quotient). Clearly $U_q(\mathfrak{l}) /
  M'' = \bigoplus_\gamma (U_q(\mathfrak{l}) / M'')_\gamma$. Let $N$ be
  a submodule of $U_q(\mathfrak{l}) / M''$. Then $N = \bigoplus_\gamma
  N \cap (U_q(\mathfrak{l}) / M'')_\gamma$. Since
  $(U_q(\mathfrak{l})/M'')_\lambda = (C_q / M) \iso (V_{1})_\lambda$
  is a simple $C_q$-module we have either $N\cap
  (U_q(\mathfrak{l})/M'')_\lambda = (U_q(\mathfrak{l})/M'')_\lambda$
  or $N \cap (U_q(\mathfrak{l})/M'')_\lambda = 0$. In the first case
  we have $1+M'' \in N$ and so $N = U_q(\mathfrak{l})/M''$. So all
  proper submodules of $U_q(\mathfrak{l})/M''$ have $N \cap
  (U_q(\mathfrak{l})/M'')_\lambda = 0$. Let $N_0$ be the sum of all
  proper submodules. Then this is the unique maximal submodule of
  $U_q(\mathfrak{l})/M''$.  So there is a unique maximal submodule
  $M'$ of $U_q(\mathfrak{l})$ containing $M$.

  Set $M_i = \Ann_{C_q}(v_i)$. Then from the above we get unique
  maximal left ideals $M_i'$ of $U_q(\mathfrak{l})$ containing
  $M_i$. By the uniqueness we have $M_i' =
  \Ann_{U_q(\mathfrak{l})}(v_i)$ and we have $V_i \iso
  U_q(\mathfrak{l}) / M_i'$. Let $\phi: C_q/M_1 \to C_q / M_2$ be the
  isomorphism between $(V_1)_\lambda$ and $(V_2)_\lambda$ and suppose
  $\phi(1+M_1)=x+M_2$. Then define $\Phi:U_q/M_1' \to U_q / M_2'$ by
  $\Phi(u + M_1') = ux + M_2'$. Then $\Phi$ is a
  $U_q(\mathfrak{l})$-isomorphism because $\Phi$ is a nonzero
  homomorphism between two simple modules.
\end{proof}

\begin{lemma}
  \label{lemma:3}
  Let $\lambda\in X$. Let $N$ be a simple $C_q$-module such that
  $K_\alpha n = \lambda(K_\alpha)n$, for all $\alpha\in \Pi$ and $n\in
  N$. Then there exists a simple $U_q(\mathfrak{l})$-module $V$ such
  that $N\iso V_\lambda$ as a $C_q$-module.
\end{lemma}
\begin{proof}
  Let $0\neq n \in N$ and set $M=\Ann_{C_q}(n)$. Then there exists a
  maximal left ideal $M'$ of $U_q(\mathfrak{l})$ like in the proof of
  Lemma~\ref{lemma:2}. Set $V= U_q(\mathfrak{l})/M'$. This is a simple
  module since $M'$ is maximal. We claim that $V_\lambda \iso N$ as
  $C_q$-modules. This follows from the fact that $C_q \cap M' = M$:
  
  $M \subset C_q \cap M'$ by definition. Take any $x\in C_q \cap M'$
  and assume $x \not \in M$. Since $M$ is maximal in $C_q$ we must
  have $y\in C_q$ such that $yx - 1 \in M$ hence $1 \in M'$. This is a
  contradiction. So $M=C_q \cap M'$.
\end{proof}

It now follows that we have just like Theorem~1 in~\cite{Lemire} the
theorem:
\begin{thm}
  \label{thm:Lemire}
  Let $\lambda\in X$.  There is a $1-1$ correspondence between simple
  $U_q(\mathfrak{l})$-modules $V$ with weight $V_\lambda\neq 0$ and
  simple $C_q$ modules with weight $\lambda$ given by: For $V$ a
  $U_q(\mathfrak{l})$-module, $V_\lambda$ is the corresponding simple
  $C_q$-module.
\end{thm}

The next lemma we will prove is the equivalent of Lemma~4.5
in~\cite{Fernando}. The proof goes in almost exactly the same way.

\begin{lemma}
  \label{lemma:4}
  Let $L$ be a simple $U_q(\mathfrak{l})$-module. Let
  $U_q(\mathfrak{t})$ and $U_q(\tau)$ be defined as above. There
  exists a simple $U_q(\tau)$-module $L_1$ and a simple
  $U_q(\mathfrak{t})$-module $L_2$ such that $L \iso L_1 \tensor_\setC
  L_2$ as a $U_q(\mathfrak{l})=U_q(\tau)\tensor_\setC
  U_q(\mathfrak{t})$ module. Furthermore if
  $\Pi_{T_L^s}'=\bigcup_{i=1}^s \Pi_{(T_L^s)_i}'$ where
  $\Pi_{(T_L^s)_i}'$ are the different connected components in
  $\Pi_{T_L^s}'$ set $\mathfrak{t}_i = \mathfrak{g}_{(T_L)_i} \oplus
  \mathfrak{h}_{(T_L)_i}$ and $U_q(\mathfrak{t}_i) = \left<
    F_\alpha,K_\alpha,E_\alpha \right>_{\alpha\in
    \Pi_{(T_L^s)_i}'}$. Then $U_q(\mathfrak{t}) \iso
  U_q(\mathfrak{t}_1)\tensor_\setC \cdots \tensor_\setC
  U_q(\mathfrak{t}_s)$ and there exists simple
  $U_q(\mathfrak{t}_i)$-modules $(L_2)_i$ such that $L_2 \iso (L_2)_1
  \tensor_\setC \cdots \tensor_\setC (L_2)_s$ as
  $U_q(\mathfrak{t}_1)\tensor_\setC \cdots \tensor_\setC
  U_q(\mathfrak{t}_s)$-modules.
\end{lemma}
\begin{proof}
  Let $\lambda$ be one of the weights of $L$. Then we know that $E:=
  L_\lambda$ is a simple finite dimensional $C_q$-module. Let $R$
  (respectively $R_1$ and $R_2$) denote the image of $C_q$
  (respectively $U_q(\tau)_0$ and $U_q(\mathfrak{t})_0$) in
  $\End_{\setC}(E)$. Since $E$ is simple we have
  $R=\End_{\setC}(E)$. Since $R_1 E \neq 0$ there exists a nontrivial
  $R_1$-submodule of $\res_{R_1}^R E$ and since $E$ is finite
  dimensional there exists a \emph{simple} $R_1$-submodule $E_1$ of
  $\res_{R_1}^R E$. The simplicity of $E_1$ implies that the
  representation $R_1\to \End_\setC(E_1)$ is surjective. The kernel of
  $R_1 \to \End_\setC(E_1)$ must be $\Ann_{R_1}(E_1)$. But if this is
  nonzero then since $E = R E_1 = R_2 E_1$ and since $R_1$ and $R_2$
  commutes we see that $\Ann_{R}(E)$ will be nonzero which is a
  contradiction since $R=\End_\setC(E)$. So $R_1 \iso \End_\setC(E_1)$
  is simple. Similarly there exists a simple $R_2$-module $E_2$ and
  $R_2\iso \End_\setC(E_2)$ is simple. Now as in the proof of
  Lemma~4.5 in~\cite{Fernando} we get $R\iso R_1 \tensor R_2$
  (using~\cite[Theorem~7.1D]{MR0010543}). Since $R=\End_\setC(E)$ it
  has exactly one simple module up to isomorphism. This implies that
  $E \iso E_1\tensor_\setC E_2$ as $R$-modules.

  Now set $L_1 = U_q(\tau) E_1$ and $L_2 = U_q(\mathfrak{t}) E_2$. We
  have $L_\lambda = E \iso E_1\tensor_\setC E_2 = (L_1\tensor_\setC
  L_2)_\lambda$ and by Theorem~\ref{thm:Lemire} this implies that
  $L\iso L_1\tensor_\setC L_2$.

  The second part of the lemma is proved in the same way. The only
  thing we used about $U_q(\tau)$ and $U_q(\mathfrak{t})$ was that
  $U_q(\mathfrak{l}) = U_q(\tau)U_q(\mathfrak{t})$ and that
  $U_q(\tau)_0$ and $U_q(\mathfrak{t})_0$ commutes. The same is true
  for $U_q(\mathfrak{t})$ and the $U_q(\mathfrak{t}_i)$'s.
\end{proof}

To summarize we have the following equivalent of Theorem~4.18
in~\cite{Fernando}:

\begin{thm}
  \label{thm:classification}
  Suppose $L\in\mathcal{F}$ is a simple $U_q(\mathfrak{g})$
  module. Let $w\in W$ be such that $P_{{^w}L}$ is standard
  parabolic. With notation as above: $({^w}L)^\mathfrak{u}$ is a
  simple $U_q(\mathfrak{l})$-module and this module decomposes into a
  tensor product $X_{\operatorname{fin}}\tensor_\setC
  X_{\operatorname{fr}}$ where $X_{\operatorname{fin}}$ is a finite
  dimensional simple $U_q(\tau)$-module and $X_{\operatorname{fr}}$ is
  a torsion free $U_q(\mathfrak{t})$-module. Furthermore if
  $\mathfrak{t}= \mathfrak{t}_1\oplus \cdots \oplus \mathfrak{t}_s$ as
  a sum of ideals then $X_{\operatorname{fr}} = X_1\tensor_\setC
  \cdots \tensor_\setC X_s$ for some simple
  $U_q(\mathfrak{t}_i)$-modules.

  Given the pair $(X_{\operatorname{fin}},X_{\operatorname{fr}})$ and
  the $w\in W$ defined above then $L$ can be recovered as
  ${^{\bar{w}}}L(X_{\operatorname{fin}}\tensor_\setC\nobreak
  X_{\operatorname{fr}}))$.
\end{thm}
So the problem of classifying simple modules in $\mathcal{F}$ is
reduced to the problem of classifying finite dimensional simple
modules of $U_q(\tau)$ and classifying torsion free simple modules of
$U_q(\mathfrak{t})$ where $\mathfrak{t}$ is a simple Lie algebra. In
the next section we will show that we can make the same reduction if
$q$ is an odd root of unity. The procedure is similar but there are
some differences, e.g. because the $\mathfrak{sl}_2$ theory is a
little different.

\section{Root of unity case: Reduction}
\label{sec:root-unity-case}
We will now consider the root of unity case. In this section
$q\in\setC$ will be assumed to be a primitive $l$'th root of unity
where $l$ is odd.

\begin{lemma}
  \label{lemma:23}
  Let $\lambda\in X$ and $\alpha\in \Pi$. Then $\lambda(K_\alpha) =
  \pm q_\alpha^{k}$ for some $k\in \{0,\dots,l-1\}$.
\end{lemma}
\begin{proof}
  By Section 6.4 in~\cite{MR1066560} we have the following relation in
  $U_A$:
  \begin{equation*}
    {K_\alpha; 0 \brack l-1}{K_\alpha;-l+1 \brack 1} = {l \brack l-1}_{v_\alpha} {K_\alpha;0 \brack l}.
  \end{equation*}
  Since ${l \brack l-1}_{q_\alpha} = 0$ when $q$ is an $l$'th root of
  unity we must have that either
  $q_\alpha^{-l+1}\lambda(K_\alpha)-q_\alpha^{l-1}\lambda(K_\alpha)\inv=0$
  or
  $q_\alpha^{1-k}\lambda(K_\alpha)-q_\alpha^{k-1}\lambda(K_\alpha)\inv=0$
  for some $k\in\{1,\dots,l-1\}$. Writing out what these equations
  imply we get that $\lambda(K_\alpha) = \pm q_\alpha^k$ for some
  $k\in\{0,\dots,l-1\}$.
\end{proof}

\begin{defn}
  \begin{equation*}
    \Lambda_l=\{\lambda \in \Lambda| 0\leq
    \left<\lambda,\alpha^\vee \right> < l, \, \forall \alpha\in \Pi\}
  \end{equation*}
\end{defn}

\begin{lemma}
  \label{lemma:25}
  Let $\lambda:U_q^0 \to \setC$ be an algebra homomorphism. Then
  $\lambda$ is completely determined by its values on $K_{\alpha}$ and
  ${K_{\alpha};0\brack l}$ with $\alpha \in \Pi$.  Choosing a
  homomorphism $\sigma:Q \to \{\pm 1\}$, an element $\lambda^0 \in
  \Lambda_l$ and an element $\lambda^1 \in \mathfrak{h}^*$ determines
  a homomorphism $\lambda\in X$ as follows: For $\alpha\in \Pi$:
  \begin{align*}
    \lambda(K_\alpha) =& \sigma(\alpha)q^{( \lambda^0|\alpha)}
    \\
    \lambda({K_\alpha; 0 \brack l}) =& \left< \lambda^1,\alpha^\vee
    \right>.
  \end{align*}
  All algebra homomorphisms $\lambda:U_q^0 \to \setC$ are of this form,
  i.e. $X = S \times \Lambda_l \times \mathfrak{h}^*$ in this case,
  where $S$ is the set of homomorphisms $\sigma:Q\to \{\pm 1\}$.
\end{lemma}
\begin{proof}
  We will use the relations for $U_A$ from Section~6.4
  of~\cite{MR1066560}. Let $\beta\in \Pi$. If $\lambda(K_\beta)=d$
  then $\lambda(K_\beta\inv)=d\inv$ and the value on ${K_\beta;c\brack
    t} = \prod_{i=1}^t
  \frac{q_\beta^{c-i+1}K_\beta-q_\beta^{i-1-c}K_\beta\inv}{q_\beta^i -
    q_\beta^{-i}}$ for $0\leq t< l$ is also determined. The relations
  \begin{equation*}
    {K_\beta ; c \brack l} - {K_\beta ; c+1 \brack l} = -q_\beta^{c+1}K_\beta {K_\beta; c \brack l-1}
  \end{equation*}
  determine the values on ${K_\beta ; c \brack t}$ for all $c\in
  \setZ$ if the value on ${K_\beta ; 0 \brack t}$ and the value on
  $K_\beta$ is known.  Finally if $c=rl+t$ with $0\leq t < l$ we have
  \begin{align*}
    {K_\beta; 0 \brack rl + t} =& {K_\beta; 0 \brack
      rl}{K_\beta;-rl\brack t}
    \\
    =& r\inv {K_\beta;0 \brack (r-1)l}{K_\beta;-(r-1)l \brack
      l}{K_\beta;-rl\brack t}
    \\
    \vdots&
    \\
    =& (r!)\inv \prod_{s=0}^{r-1} {K_\beta;-sl \brack l}
    {K_\beta;-rl\brack t}.
  \end{align*}
  So determining the value on $K_\beta$ and ${K_\beta;0\brack l}$
  determines the value on all of $U_q^0$.
  
  If $\sigma,\lambda^0,\lambda^1$ is chosen as above it is easy to
  check that the relations from Section~6.4 in~\cite{MR1066560} are
  satisfied. That all characters are of this form follows from
  Lemma~\ref{lemma:23}.
\end{proof}
It can be noted in the above that $\lambda^1 = \lambda \circ
\operatorname{Fr}'|_{\mathfrak{h}}$ where
$\operatorname{Fr}':U(\mathfrak{g})\to
U_q(\mathfrak{g})/\left<K_\alpha^l-1|\alpha\in \Pi\right>$ is the
Frobenius map from~\cite{KL2}.  We will restrict to modules of type
$\mathbf{1}$ meaning $\sigma(\alpha)=1$ for all $\alpha\in \Pi$ in the
above. It is standard how to get from modules of type $\mathbf{1}$ to
modules of any other type $\sigma$
(cf. e.g.~\cite[Section~5.1-5.4]{Jantzen}).

Since we restrict to modules of type $1$ we will assume from now on
that $X = \Lambda_l \times \mathfrak{h}^*$. A weight $\lambda\in X$
will also be written as $(\lambda^0,\lambda^1)\in \Lambda_l\times
\mathfrak{h}^*$.

\begin{lemma}
  \label{lemma:28}
  Let $\lambda\in X$ with $\lambda^0$ and $\lambda^1$ defined as in
  Lemma~\ref{lemma:25}. Let $\beta\in \Phi^+$, $c\in \setZ$,
  \begin{equation*}
    \lambda({K_\beta ; c+1 \brack l})=
    \begin{cases}
      \lambda({K_\beta ; c \brack l})+1 &\text{ if } \left<
        \lambda^0,\beta^\vee \right>+c \equiv -1 \mod l
      \\
      \lambda({K_\beta ; c \brack l}) &\text{ otherwise }.
    \end{cases}
  \end{equation*}
\end{lemma}
\begin{proof}
  Set $a=\left< \lambda^0,\beta^\vee\right>$. By $(b4)$ in Section~6.4
  of~\cite{MR1066560}
  \begin{align*}
    \lambda\left({K_\beta ; c\brack l}\right) =& \lambda\left(
      {K_\beta ; c-1\brack l} + q_\beta^{c} K_\beta {K_\beta ; c-1
        \brack l-1}\right)
    \\
    =& \lambda\left( {K_\beta ; c-1\brack l}\right) + q_\beta^{c+a}
    {a+c-1 \brack l-1}_{q_\beta}.
  \end{align*}
  ${a+c-1 \brack l-1}_{q_\beta}$ is zero unless $a+c-1 \equiv -1 \mod
  l$. If $a+c-1\equiv -1 \mod l$ then $a+c \equiv 0 \mod l$ and so
  $q_\beta^{a+c}=1$ and ${a+c-1\brack l-1}_{q_\beta} = {l-1 \brack
    l-1}_{q_\beta} = 1$.
\end{proof}

For a character $\lambda\in X$ and a $\mu\in Q$ we define
$q^\mu \lambda$ as follows:
\begin{align*}
  (q^\mu \lambda)(K_\alpha) =& q^{(\mu|\alpha)}\lambda(K_\alpha) =
  q_\alpha^{\left<\mu,\alpha^\vee\right>}\lambda(K_\alpha)
  \\
  (q^\mu \lambda)\left({K_\alpha ; c \brack l}\right) =&
  \lambda\left({K_\alpha ; c+\left<\mu,\alpha^\vee\right> \brack
      l}\right).
\end{align*}
With this notation we get for a module $M$ that $E_\alpha^{(r)}
M_\lambda \subset M_{q^{r \alpha}\lambda}$ and $F_\alpha^{(r)}
M_\lambda \subset M_{q^{-r\alpha}\lambda}$. Note also that
$(q^{l\beta}\lambda)^1= \lambda^1 + \beta$.

We use the same definitions as in
Section~\ref{sec:nonroot-unity-case}:

\begin{defn}
  Let $M\in \mathcal{F}$ and let $\beta\in \Phi$. We call $M$
  $\beta$-finite if $q^{\setN \beta}\lambda\cap \wt M$ is a finite set
  for all $\lambda\in \wt M$ where $q^{\setN \beta}\lambda =
  \{q^{r\beta}\lambda|r\in \setN\}$.
\end{defn}

The weight vectors $E_\beta$ and $F_\beta$ for positive $\beta$ that
are not simple are defined just as before by choosing a reduced
expression of $w_0$. By~\cite[Section~5.6]{MR1066560} the divided
powers $E_\beta^{(r)}:= \frac{1}{[r]_\beta!}E_\beta^r$, $r\in\setN$
are all contained in $U_A$ and by abuse of notation we use the same
symbol for the corresponding elements in $U_q$.

\begin{prop}
  \label{prop:10}
  Let $M\in\mathcal{F}$ and let $\beta$ be a positive root. Let
  $E_\beta$ be any choice of root vector corresponding to
  $\beta$. Then the following are equivalent:
  \begin{enumerate}
  \item $M$ is $\beta$-finite.
  \item For all $m\in M$, $E_\beta^{(r)}m = 0$ for $r>>0$
  \end{enumerate}
\end{prop}
\begin{proof}
  Clearly $1.$ implies $2.$ since $E_\beta^{(r)}M_\lambda \subset
  M_{q^{r\beta}\lambda}$. Assume $2.$ and suppose $M$ is not $\beta$
  finite.

  We must have a $\lambda\in \wt M$, an increasing sequence
  $\{j_i\}_{i\in\mathbb{N}}$, weights $\mu_i = q^{j_i \beta}\lambda
  \in \wt M$ and weight vectors $m_{i}\in M_{\mu_i}$ such that
  $E_\beta^{(r)} m_i = 0$ for all $r\in\setN\backslash\{0\}$. We can
  assume without loss of generality that if
  $\lambda\left({K_{\beta};0\brack l}\right) \in \setZ$ then
  $\lambda\left({K_{\beta};0\brack l}\right) \in \setZ_{>0}$ by
  Lemma~\ref{lemma:28}.

  Now consider the subalgebra $D_\beta$ of $U_q$ generated by
  $E_\beta^{(r)}$, $K_\beta^{\pm 1}$ and $F_\beta^{(r)}$ for
  $r\in\setN$ where $F_\beta$ is the root vector corresponding to
  $E_\beta$ (i.e. if $E_\beta = T_w(E_{\alpha_i})$ then $F_\beta =
  T_w(F_{\alpha_i})$). For each $i$ we get a $D_\beta$-module $D_\beta
  m_i$ with highest weight $\mu_i$. We claim that in each of these
  modules we have at least one weight vector with one of the weights
  $\lambda,q^{-\beta} \lambda,\dots, q^{-(l-1)\beta}\lambda$. So we
  want to show for each $m_i$ that at least one of the vectors
  $F_\beta^{(j_i)}m_i,F_{\beta}^{(j_i+1)}m_i,\dots,F_\beta^{(j_i+l-1)}m_i$
  is nonzero. We must have that one of the numbers $j_i,\dots,j_i+l-1$
  is congruent to $0$ modulo $l$. Lets call this number $k$. Say
  $k=rl$. Now we have
  \begin{align*}
    E_\beta^{(k)} F_\beta^{(k)} m_i =& \sum_{s\geq 0} F_\beta^{(k-s)}
    { K_\beta ; 2s - 2k \brack s} E_\beta^{(k-s)} m_i
    \\
    =& {K_\beta ; 0 \brack rl} m_i
    \\
    =& \frac{1}{r!}\prod_{s=0}^{r-1} {K_\beta ; -sl \brack l} m_i
    \\
    =& \frac{1}{r!}\prod_{s=0}^{r-1} \left(c_i -s\right)m_i
    \\
    =& {c_i \choose r} m_i
  \end{align*}
  where $c_i = \mu_i\left({K_\beta;0 \brack l}\right)$. To show that
  this is nonzero we must show that $c_i \not\in \{0,\dots,r-1\}$. If
  $\lambda\left( {K_\beta;0\brack l}\right)$ is not an integer then
  this is automatically fullfilled. Otherwise we know $j_i = rl-t$ for
  some $t=0,\dots,l-1$. So $\mu_i = q^{(rl-t)\beta}\lambda$ and by
  Lemma~\ref{lemma:28}
  \begin{equation*}
    c_i = \mu_i\left( {K_\beta;0\brack l}\right) = q^{(rl-t)\beta}\lambda\left( {K_\beta;0\brack l}\right) = \lambda \left( {K_\beta;0\brack l}\right) + r-1 \geq r.
  \end{equation*}

  Since there are infinitely many $m_i$'s we must have infinitely many
  weight vectors $\{v_j\}$ of weight one of the weights
  $\lambda,\lambda-\beta,\dots, \lambda-(l-1)\beta$.
  
  To show that they are linearly independent let $v_1,\dots, v_n$ be a
  finite set of the above weight vectors. They are all of the form
  $F_\beta^{(k_i)}m_i$ for some $i$ and some $k_i$. Assume $v_n$ is
  the vector where the power $k_n$ is maximal. Then
  $E_\beta^{(k_n)}v_i=0$ for $i\neq n$ and $E_\beta^{(k_n)}v_n \neq
  0$. It follows by induction on $n$ that the set $\{v_1,\dots,v_n\}$
  is linearly independent.
\end{proof}

We define $M^{[\beta]} = \{m\in M| \dim \left< E_\beta^{(r)} |
  r\in\setN\right> m < \infty \}$.  Proposition~\ref{prop:2} and
Lemma~\ref{lemma:5} carry over with the same proof. In particular
$M^{[\beta]}$ is independent of the choice of root vector
$E_\beta$. Again we call $M$ $\beta$-free if $M^{[\beta]}=0$. Again we
can show everything with $F$'s instead of $E$'s if $\beta$ is
negative.

Propositions~\ref{prop:3}~and~\ref{prop:4} carry over with almost
identical proofs. Setting $l=1$ in the propositions and their proofs
below would make the proofs identical.

\begin{prop}
  \label{prop:13}
  Let $M\in\mathcal{F}$ be a simple module and $\beta$ a root. Then
  $M$ is $\beta$-free if and only if $q^{\setN l\beta}\wt M \subset
  \wt M$.
\end{prop}
\begin{proof}
  Assume $\beta$ is positive.  If $q^{\setN l \beta}\wt M \subset \wt
  M$ then $M$ is clearly not $\beta$-finite and since $M$ is simple we
  have by Proposition~\ref{prop:2} that $M$ is $\beta$-free in this
  case. For the other way assume $M$ is $\beta$-free and assume we
  have a weight vector $0\neq m \in M_\lambda$ such that
  $E_\beta^{(rl)}m = 0$ for some $r\in\setN$. For any $i\in\setN$,
  ${i+rl \brack i}_\beta\neq 0$ so
  \begin{equation*}
    E_\beta^{(rl+i)} m= {i+rl \brack i}_\beta \inv E_\beta^{(i)} E_\beta^{(rl)} m = 0
  \end{equation*}
  But this implies that $m\in M^{[\beta]}$ which contradicts the
  assumption that $M$ is $\beta$-free.  If $\beta$ is negative we do
  the same with $F$'s instead of $E$'s.
\end{proof}

\begin{prop}
  \label{prop:14}
  Let $L\in\mathcal{F}$ be a simple module. $T_L$ is a closed subset
  of $\Phi$.
\end{prop}
\begin{proof}
  Assume $\beta,\gamma\in T_L$ with $\beta+\gamma\in \Phi$. Then since
  $\beta\in T_L$, $q^{\setN l \beta}\wt L\subset \wt L$. Since
  $\gamma\in T_L$ we get then $q^{\setN l\gamma}q^{\setN l\beta}\wt L
  \subset \wt L$ so $q^{\setN l(\beta+\gamma)} \wt L \subset \wt L$.
\end{proof}

\begin{prop}
  \label{prop:12}
  Let $L\in\mathcal{F}$ be a simple module. $F_L$ and $T_L$ are closed
  subsets of $\Phi$ and $\Phi=F_L\cup T_L$ (disjoint union).
\end{prop}
\begin{proof}
  $T_L$ is closed by Proposition~\ref{prop:14}. $F_L$ is closed by the
  same proof as the proof of Proposition~\ref{prop:5}. Note that the
  constants in the proof of Proposition~\ref{prop:5} that are inverted
  are all nonzero even when $q$ is a $l$'th root of unity as long as
  $l$ is odd.
\end{proof}

We define $P_L$ like in Section~\ref{sec:nonroot-unity-case} and we
assume like above that $P_L$ is standard parabolic by considering
${^w}L$ for an appropiate $w\in W$.  The subalgebras
$U_q(\mathfrak{p})$, $U_q(\mathfrak{l})$, $U_q(\mathfrak{u})$,
$U_q(\mathfrak{u}^-)$ etc. are defined as above but this time with
divided powers. For example we have
\begin{equation*}
  U_q(\mathfrak{p}) = \left< E_{\beta_j}^{(r)},K_\mu,F_{\beta_i^1}^{(r)}\right>_{j=1,\dots,N,\mu\in Q, i=1,\dots, h, r\in\setN}
\end{equation*}
and so on. Now the rest of the lemmas and proposition carry over with
the same proofs as before and we have the following equivalent of
Theorem~\ref{thm:classification}:

\begin{thm}
  \label{thm:classification_root_of_unity}
  Suppose $L\in\mathcal{F}$ is a simple $U_q(\mathfrak{g})$
  module. Let $w\in W$ be such that $P_{{^w}L}$ is standard
  parabolic. With notation as above: $({^w}L)^\mathfrak{u}$ is a
  simple $U_q(\mathfrak{l})$-module and this module decomposes into a
  tensor product $X_{\operatorname{fin}}\tensor_\setC
  X_{\operatorname{fr}}$ where $X_{\operatorname{fin}}$ is a finite
  dimensional simple $U_q(\tau)$-module and $X_{\operatorname{fr}}$ is
  a torsion free simple $U_q(\mathfrak{t})$-module. Furthermore if
  $\mathfrak{t}= \mathfrak{t}_1\oplus \cdots \oplus \mathfrak{t}_s$ as
  a sum of ideals then $X_{\operatorname{fr}} = X_1\tensor_\setC
  \cdots \tensor_\setC X_s$ as $U_q(\mathfrak{t}_1)\tensor\cdots
  \tensor U_q(\mathfrak{t}_s)$-module for some simple
  $U_q(\mathfrak{t}_i)$-modules $X_i$, $i=1,\dots,s$.

  Given the pair $(X_{\operatorname{fin}},X_{\operatorname{fr}})$ and
  the $w\in W$ defined above then $L$ can be recovered as
  ${^{\bar{w}}}(L(X_{\operatorname{fin}}\tensor_\setC\nobreak
  X_{\operatorname{fr}}))$.
\end{thm}

So in the root of unity case we have also that to classify simple
modules in $\mathcal{F}$ we just have to classify finite dimensional
modules of $U_q(\tau)$ and 'torsion free' modules over
$U_q(\mathfrak{t})$, where $\mathfrak{t}$ can be assumed to be a
simple Lie algebra.

\section{$U_A$ formulas}
\label{sec:u_a-calculations}
In this section we recall from~\cite{DHP-twist} some formulas for
commuting root vectors with each other that will be used later. Note
that in~\cite{DHP-twist} the braid operators that we here call $T_w$
are denoted by $R_w$. In~\cite{DHP-twist} $T_w$ denotes twisting
functors.

Recall that $A=\setZ[v,v\inv]$ where $v$ is an indeterminate and $U_A$
is the $A$-subspace of $U_v$ generated by the divided powers
$E_{\alpha}^{(n)}$, $F_\alpha^{(n)}$, $n\in\setN$ and $K_\alpha$,
$K_\alpha\inv$.
\begin{defn}
  \label{sec:twisting-functors-2}
  Let $x\in (U_v)_\mu$ and $y\in (U_v)_\gamma$ then we define
  \[ [x,y]_v:=xy-v^{-(\mu|\gamma)}yx
  \]
\end{defn}

\begin{thm}
  \label{thm:DP}
  Suppose we have a reduced expression of $w_0 = s_{i_1}\cdots
  s_{i_N}$ and define root vectors
  $F_{\beta_1},\dots,F_{\beta_N}$. Let $i<j$. Let $A=\setZ[v,v\inv]$
  and let $A'$ be the localization of $A$ in $[2]$ if the Lie algebra
  contains any $B_n,C_n$ or $F_4$ part. Then
  \begin{equation*}
    [F_{\beta_j},F_{\beta_i}]_v=F_{\beta_j}F_{\beta_i}-v^{-(\beta_i|\beta_j)}F_{\beta_i}F_{\beta_j}\in \spa{A'}{F_{\beta_{j-1}}^{a_{j-1}}\cdots F_{\beta_{i+1}}^{a_{i+1}}}
  \end{equation*}
\end{thm}
\begin{proof}
  \cite[Proposition~5.5.2]{Levendorski-Soibelman}. Detailed proof also
  in~\cite[Theorem~2.9]{DHP-twist}.
\end{proof}

\begin{defn}
  Define $\ad(F_\beta^i)(F_\alpha)
  :=[[\dots[F_\alpha,F_\beta]_v\dots]_v,F_\beta]_v$ and
  $\tilde{\ad}(F_\beta^i)(F_\alpha)
  :=[F_\beta,[\dots,[F_\beta,F_\alpha]_v\dots]]_v$ where the
  commutator is taken $i$ times from the left and right respectively.
\end{defn}

\begin{prop}
  \label{prop:16}
  Let $u\in (U_A)_\mu$, $\beta\in \Phi^+$ and $F_\beta$ a
  corresponding root vector. Set $r=\left<\mu,\beta^\vee\right>$. Then
  in $U_A$ we have the identity
  \begin{align*}
    \ad(F_\beta^{i})(u) = [i]_\beta! \sum_{n=0}^i (-1)^{n}
    v_\beta^{n(1-i-r)} F_\beta^{(n)} u F_\beta^{(i-n)}
  \end{align*}
  and
  \begin{align*}
    \tilde{\ad}(F_\beta^{i})(u) = [i]_\beta! \sum_{n=0}^i (-1)^{n}
    v_\beta^{n(1-i-r)} F_\beta^{(i-n)} u F_\beta^{(n)}
  \end{align*}
\end{prop}
\begin{proof}
  Proposition~1.8 in~\cite{DHP-twist}.
\end{proof}
So we can define $\ad(F_\beta^{(i)})(u) := ([i]!)\inv
\ad(F_\beta^i)(u)\in U_A$ and $\tilde{\ad}(F_\beta^{(i)})(u) :=
([i]!)\inv \tilde{\ad}(F_\beta^i)(u)\in U_A$.

\begin{prop}
  \label{prop:17}
  Let $a\in\setN$, $u\in (U_A)_\mu$ and
  $r=\left<\mu,\beta^\vee\right>$. In $U_A$ we have the identities
  \begin{align*}
    u F_{\beta}^{(a)} =& \sum_{i=0}^a v_\beta^{(i-a)(r+i)}
    F_\beta^{(a-i)}\ad(F_\beta^{(i)})(u)
    \\
    =& \sum_{i=0}^{a} (-1)^i v_\beta^{a(r+i)-i} F_\beta^{(a-i)}
    \tilde{\ad}(F_\beta^{(i)})(u)
  \end{align*}

  and
  \begin{align*}
    F_{\beta}^{(a)} u =& \sum_{i=0}^a v_\beta^{(i-a)(r+i)}
    \tilde{\ad}(F_\beta^{(i)})(u)F_\beta^{(a-i)}
    \\
    =& \sum_{i=0}^{a} (-1)^i v_\beta^{a(r+i)-i}
    \ad(F_\beta^{(i)})(u)F_\beta^{(a-i)}
  \end{align*}
\end{prop}
\begin{proof}
  Proposition~1.9 in~\cite{DHP-twist}.
\end{proof}

\begin{prop}
  \label{derivative}
  For $x_1\in (U_A)_{\mu_1}$, $x_2\in (U_A)_{\mu_2}$ and $y\in
  (U_A)_\gamma$ we have
  \begin{equation*}
    [y,x_1x_2]_v=x_1[y,x_2]_v+v^{-(\gamma|\mu_2)}[y,x_1]_vx_2
  \end{equation*}
  and
  \begin{equation*}
    [x_1x_2,y]_v=v^{-(\gamma|\mu_1)}x_1[x_2,y]_v+[x_1,y]_vx_2
  \end{equation*}
\end{prop}
\begin{proof}
  Direct calculation.
\end{proof}

Let $s_{i_1}\dots s_{i_N}$ be a reduced expression of $w_0$ and
construct root vectors $F_{\beta_i}$, $i=1,\dots,N$. In the next lemma
$F_{\beta_i}$ refers to the root vectors constructed as such. In
particular we have an ordering of the root vectors.

\begin{lemma}
  \label{lemma:22}
  Let $n\in \setN$. Let $1\leq j<k\leq N$.
  
  $\ad(F_{\beta_j}^{(i)})(F_{\beta_k}^{(n)})=0$ and
  $\tilde{\ad}(F_{\beta_k}^{(i)})(F_{\beta_j}^{(n)})=0$ for $i\gg 0$.
\end{lemma}
\begin{proof}
  Lemma~1.11 in~\cite{DHP-twist}.
\end{proof}

\section{Ore localization and twists of localized modules}
\label{Torsion_free_modules_root_of_unity}

In this section $q$ will be a complex primitive $l$'th root of unity
with $l$ odd. Recall that we will assume $X = \Lambda_l \times
\mathfrak{h}^*$ in this case restricting to modules of type $1$. For
an element $\lambda \in X$ we define $\lambda^0\in \Lambda_l$ and
$\lambda^1\in \mathfrak{h}^*$ as in Lemma~\ref{lemma:25} such that
$\lambda(K_\alpha) = q^{\left( \lambda^0|\alpha \right) }$ and
$\lambda({K_\alpha ; 0 \brack l}) = \left< \lambda^1 , \alpha^\vee
\right>$ for $\alpha\in \Pi$ and we will also write $\lambda =
(\lambda^0,\lambda^1)\in X$.

\begin{lemma}
  \label{lemma:19}
  Let $\beta$ be a positive root and $F_\beta$ a corresponding root
  vector.  The set
  \begin{equation*}
    \{r! F_\beta^{(rl)}|r\in \setN\}=\{
    \left(F_\beta^{(l)}\right)^r|r\in \setN\}
  \end{equation*}
  is an Ore subset of $U_q$.
\end{lemma}
\begin{proof}
  We can assume $\beta$ is simple since otherwise
  $F_\beta=T_w(F_\alpha)$ for some $\alpha\in \Pi$ and some $w\in W$
  and $T_w(U_q)=U_q$.  Since $r!F_\beta^{(rl)} k!F_\beta^{(kl)} =
  (r+k)! F_\beta^{(rl+kl)}$ the set is multiplicative and does not
  contain $0$. We will show the Ore property for each generator of
  $U_q$. First consider $\alpha\in \Pi$ a simple root not equal to
  $\beta$. Let $n\in \setN$. We have the following identities for
  $r\geq 1$ (cf. Proposition~\ref{prop:17})
  \begin{align*}
    r! F_\beta^{(rl)} E_\alpha^{(n)} =& E_\alpha^{(n)} r!
    F_\beta^{(rl)}
    \\
    r! F_\beta^{(rl)} K_\alpha^{\pm 1} =& K_\alpha^{\pm 1} r!
    F_\beta^{(rl)}
    \\
    F_\alpha^{(n)} r! F_\beta^{(rl)} = &
    r!F_\beta^{(rl)}F_\alpha^{(n)}
    \\
    &+ \sum_{k=0}^{r-1} \sum_{i=kl+1}^{kl+l} c_i
    (r-k-1)!F_\beta^{(rl-kl-l)} F_\beta^{(kl+l-i)}
    \ad(F_\beta^{(i)})(F_\alpha^{(n)} )
  \end{align*}
  where $c_i = q_\beta^{i(i+\left<\alpha,\beta^\vee
    \right>)}r(r-1)\cdots (r-k)$. Finally we have the
  $\mathfrak{sl}_2$ identities for $0\leq i \leq l$:
  \begin{align*}
    r! F_\beta^{(rl)}F_\beta^{(n)} =& F_\beta^{(n)} r! F_\beta^{(rl)}
    \\
    E_\beta^{(i)} r! F_\beta^{(rl)} =& r! F_\beta^{(rl)}E_\beta^{(i)}+
    \sum_{t=1}^{i} r (r-1)! F_\beta^{(rl-l)}F_\beta^{(l-t)}
    E_\beta^{(i-t)} {K_\beta ; i - rl \brack t }_\beta .
  \end{align*}
  
  So we have shown that it is an Ore set.
\end{proof}

We will denote the Ore localization of $U_q$ in the above set by
$U_{q(F_\beta)}$. For a $U_q$-module $M$ we define $M_{(F_\beta)} :=
U_{q(F_\beta)}\tensor_{U_q} M$. We write the inverse of
$F_\beta^{(rl)}$, $r\in\setN$ as $F_\beta^{(-rl)}$
i.e. $F_\beta^{(-rl)} = r! \left( r!  F_\beta^{(rl)} \right)\inv\in
U_{q(F_\beta)}$.

\begin{lemma}
  \label{lemma:17}
  Let $\lambda=(\lambda^0,\lambda^1) \in \Lambda_l\times
  \mathfrak{h}^*$, $\beta\in \Phi^+$ and let $F_\beta$ be a
  corresponding root vector. Let $I_\lambda$ be the left
  $U_{q(F_\beta)}$-ideal $U_{q(F_\beta)}\{(u - \lambda(u))|u\in
  U_{q}^0\}$. Then there exists, for each $b\in\setC$, an automorphism
  of $U_{q(F_\beta)}$-modules $\psi_{F_\beta,b}^\lambda :
  U_{q(F_\beta)} /I_\lambda \to
  U_{q(F_\beta)}/I_{(\lambda^0,\lambda^1+b\beta)}$ such that for $u\in
  U_{q(F_\beta)}$ and $i\in\setN$, $\psi_{F_\beta,i}^\lambda
  (u+I_\lambda)=F_\beta^{(-il)} u
  F_\beta^{(il)}+I_{(\lambda^0,\lambda^1+i\beta)}$ and the map
  $b\mapsto \psi_{F_\beta,b}^\lambda(u+I_\lambda)$ is polynomial in
  $b$. Furthermore $\psi_{F_\beta,b'}^{(\lambda^0,\lambda^1+b \beta)}
  \circ \psi_{F_\beta,b}^{\lambda} = \psi_{F_\beta,b+b'}^\lambda$ for
  $b,b'\in\setC$.
\end{lemma}
\begin{proof}
  If $\beta$ is not simple then $F_\beta = T_w(F_\alpha)$ for some
  simple root $\alpha\in \Pi$. Then we define
  $\psi_{F_\beta,b}^\lambda (u) = T_w(\psi_{F_\alpha,b}^{w \lambda}
  (T_w\inv(u)))$ where $T_w\inv(F_\beta^{(-l)}) = F_\alpha^{(-l)}$ and
  $T_w(F_\alpha^{(-l)})= F_\beta^{(-l)}$. So we assume from now on
  that $\beta\in \Pi$.

  We define $\psi_{F_\beta,b}^\lambda$ on generators: For $\alpha\in
  \Pi\backslash\{\beta\}$ and $n\in \setN$
  \begin{align*}
    \psi_{F_\beta,b}^\lambda(E_\alpha^{(n)} ) = & E_\alpha^{(n)}
    \\
    \psi_{F_\beta,b}^\lambda (F_\alpha^{(n)}) =& F_\alpha^{(n)}
    \\
    &- \sum_{k\geq 0} {b \choose k+1} \sum_{i=kl+1}^{kl+l}
    q_\beta^{i(\left<\alpha|\beta^\vee\right>-i)}
    F_\beta^{(-kl-l)}F_\beta^{(kl+l-i)}
    \ad(F_\beta^{(i)})(F_\alpha^{(n)})
    \\
    \psi_{F_\beta,b}^\lambda (K_\alpha^{\pm 1}) =&
    \lambda(K_\alpha^{\pm 1})
    \\
    \psi_{F_\beta,b}^\lambda (F_\beta^{(n)})=&F_\beta^{(n)}
    \\
    \psi_{F_\beta,b}^\lambda (E_\beta) =& E_\beta + b
    F_\beta^{(-l)}F_\beta^{(l-1)}
    [\left<\lambda_0,\beta^\vee\right>+1]_\beta
    \\
    \psi_{F_\beta,b}^\lambda(E_\beta^{(l)}) = & E_\beta^{(l)} + b
    F_\beta^{(-l)} \sum_{t=1}^{l-1} F_\beta^{(l-t)} E_\beta^{(l-t)}
    {\left<\lambda_0,\beta^\vee\right> \brack t}_\beta + b
    F_\beta^{(-l)}(\left<\lambda_1,\beta^\vee\right> +1- b)
    \\
    \psi_{F_\beta,b}^\lambda (K_\beta^{\pm 1}) =& \lambda(K_\beta^{\pm
      1}).
  \end{align*}
  The sum given in the formula for $F_\alpha^{(n)}$ is finite by
  Lemma~\ref{lemma:22}. It is easy to check that
  $\psi_{F_\beta,i}^\lambda(u+I_\lambda)=F_\beta^{(-il)} u
  F_\beta^{(il)}+I_{(\lambda^0,\lambda^1+b\beta)}$ for $i\in\setN$ and
  it is seen from the formulas that $b\mapsto
  \psi_{F_\beta,b}^\lambda(u+I_\lambda)$ is polynomial. So
  $\psi_{F_\beta,b}^\lambda$ satisfies the generating relations of
  $U_q$ for $b\in \setN$ hence it satisfies the generating relations
  for all $b\in\setC$ because $\psi_{F_\beta,b}^\lambda(u)$ is
  polynomial in $b$. Similarly we can show the rest of the claims by
  using the fact that $b\mapsto \psi_{F_\beta,b}^\lambda(u)$ is
  polynomial.
\end{proof}

We will define a twist of the action of a $U_{q(F_\beta)}$-module:
\begin{defn}
  Let $M$ be a $U_{q(F_\beta)}$-module. We define $\psi_{F_\beta,b}.M$
  to be the module equal to $M$ as a vector space with action twisted
  via $\psi_{F_\beta,b}$: For $m\in M$ we denote the corresponding
  element in $\psi_{F_\beta,b}.M$ by $\psi_{F_\beta,b}.m$. Let
  $\lambda=(\lambda^0,\lambda^1)\in \wt M$ and assume $m\in
  M_\lambda$. We have a homomorphism of $U_{q(F_\beta)}$-modules
  $\pi:U_{q(F_\beta)}/I_{\lambda} \to M$ defined by sending
  $u+I_\lambda$ in $U_{q(F_\beta)}/I_{\lambda}$ to $u m$. We define
  for $u\in U_{q(F_\beta)}$:
  \begin{equation*}
    u \cdot \psi_{F_\beta,b}.m = \psi_{F_\beta,b}. \left( \pi( \psi_{F_\beta,b}^{(\lambda^0,\lambda^1-b\beta)}(\bar{u}) ) \right)
  \end{equation*}
  where $\bar{u}=u + I_{(\lambda^0,\lambda^1-b\beta)} \in
  U_{q(F_\beta)}/I_{(\lambda^0,\lambda^1-b\beta)}$.
\end{defn}

\begin{lemma}
  \label{lemma:31}
  Let $M$ be a $U_{q(F_\beta)}$-module. Let $r\in \setZ$.
  \begin{equation*}
    \psi_{F_\beta,r}.M \iso M.
  \end{equation*}
  Furthermore for $\lambda = (\lambda^0,\lambda^1)\in \wt M$ we have
  as $(U_{q(F_\beta)})_0$-modules
  \begin{equation*}
    \psi_{F_\beta,r}.M_\lambda \iso M_{(\lambda^0,\lambda^1 - r\beta)}.
  \end{equation*}
\end{lemma}
\begin{proof}
  The isomorphism in both cases is given by $\psi_{F_\beta,r}.m
  \mapsto F_\beta^{(rl)}m$. Using the fact that
  $\psi_{F_\beta,r}^{(\lambda^0,\lambda^1-r\beta)}(u+I_{(\lambda^0,\lambda^1-r\beta)})
  = F_\beta^{(-rl)}u F_\beta^{(rl)}+I_{\lambda}$ it is easy to show
  that this is a homomorphism and the inverse is given by multiplying
  by $F_\beta^{(-rl)}$.
\end{proof}

\begin{defn}
  \label{def:commuting_roots}
  Let $\Sigma\subset \Phi^+$. Then $\Sigma$ is called a set of
  commuting roots if there exists an ordering of the roots in
  $\Sigma$; $\Sigma=\{\beta_1,\dots,\beta_s\}$ such that for some
  reduced expression of $w_0$ and corresponding construction of the
  root vectors $F_\beta$ we have: $[F_{\beta_j},F_{\beta_i}]_q=0$ for
  $1\leq i < j \leq s$.

  For any subset $I\subset \Pi$, let $Q_I$ be the subgroup of $Q$
  generated by $I$, $\Phi_I$ the root system generated by $I$ ,
  $\Phi_I^+=\Phi^+ \cap \Phi_I$ and $\Phi_I^- = -\Phi_I^+$.
\end{defn}

We have the following equivalent of Lemma~4.1 in~\cite{Mathieu}:
\begin{lemma}
  \label{lemma:7}
  \begin{enumerate}
  \item Let $I\subset \Pi$ and let $\alpha\in I$. There exists a set
    of commuting roots $\Sigma'\subset \Phi_I^+$ with
    $\alpha\in\Sigma'$ such that $\Sigma'$ is a basis of $Q_I$.
  \item Let $J,F$ be subsets of $\Pi$ with $F\neq \Pi$. Let $\Sigma'
    \subset \Phi_J^+ \backslash \Phi_{J\cap F}^+$ be a set of
    commuting roots which is a basis of $Q_J$. There exists a set of
    commuting roots $\Sigma$ which is a basis of $Q$ such that
    $\Sigma' \subset \Sigma \subset \Phi^+ \backslash \Phi_F^+$
  \end{enumerate}
\end{lemma}
\begin{proof}
  The first part of the proof is just combinatorics of the root system
  so it is identical to the first part of the proof of Lemma~4.1
  in~\cite{Mathieu}: Let us first prove assertion $2.$: If $J$ is
  empty we can choose $\alpha\in \Pi\backslash F$ and replace $J$ and
  $\Sigma'$ by $\{\alpha\}$. So assume from now on that $J\neq
  \emptyset$. Set $J' = J \backslash F$, $p= |J'|$, $q=|J|$. Let
  $J_1,\dots, J_k$ be the connected components of $J$ and set $J_i' =
  J' \cap J_i$, $F_i = F \cap J_i$, and $\Sigma_i' = \Sigma\cap
  \Phi_{J_i}$, for any $1\leq i \leq k$. Since $\Sigma'\subset \Phi_J$
  is a basis of $Q_J$, each $\Sigma_i'$ is a basis of $Q_{J_i}$. Since
  $\Sigma_i'$ lies in $\Phi_{J_i}^+ \backslash \Phi_{F_i}^+$, the set
  $J_i'=J_i\backslash F_i$ is not empty. Hence $J'$ meets every
  connected component of $J$. Therefore we can write
  $J=\{\alpha_1,\dots, \alpha_q\}$ in such a way that
  $J'=\{\alpha_1,\dots,\alpha_p\}$ and, for any $s$ with $p+1\leq
  s\leq q$, $\alpha_s$ is connected to $\alpha_i$ for some
  $i<s$. Since $\Pi$ is connected we can write $\Pi \backslash J =
  \{\alpha_{q+1},\dots, \alpha_n\}$ in such a way that, for any $s\geq
  q+1$, $\alpha_s$ is connected to $\alpha_i$ for some $i$ with $1\leq
  i < s$. So $\Pi = \{\alpha_1,\dots, \alpha_n\}$ such that for $s> p$
  we have that $\alpha_s$ is connected to some $\alpha_i$ with $1\leq
  i < s$.
  
  Let $\Sigma' = \{\beta_1,\dots,\beta_q\}$. We will define
  $\beta_{q+1},\dots,\beta_l$ inductively such that for each $s\geq
  q$, $\{\beta_1,\dots,\beta_s\}$ is a commuting set of roots which is
  a basis of $\Phi_{\{\alpha_1,\dots,\alpha_s\}}$. So assume we have
  defined $\beta_1,\dots,\beta_s$. Let $w_s$ be the longest word in
  $s_{\alpha_1},\dots,s_{\alpha_s}$ and let $w_{s+1}$ be the longest
  word in $s_{\alpha_1},\dots,s_{\alpha_{s+1}}$. Choose a reduced
  expression of $w_s$ such that the corresponding root vectors
  $\{F_\beta\}$ satisfies $[F_{\beta_j},F_{\beta_i}]_q=0$ for
  $i<j$. Choose a reduced expression of $w_{s+1}=w_{s}w'$ starting
  with the above reduced expression of $w_s$. Let $N_s$ be the length
  of $w_s$ and $N_{s+1}$ be the length of $w_{s+1}$. So we get an
  ordering of the roots generated by
  $\{\alpha_1,\dots,\alpha_{s+1}\}$:
  $\Phi_{\{\alpha_1,\dots,\alpha_{s+1}\}}^+=\{\gamma_1,\dots,\gamma_{N_s},\gamma_{N_s+1},\dots,\gamma_{N_{s+1}}\}$
  with
  $\Phi_{\{\alpha_1,\dots,\alpha_{s}\}}^+=\{\gamma_1,\dots,\gamma_{N_s}\}$. Consider
  $\gamma_{N_s+1}=w_s(\alpha_{s+1})$. Since $w_s$ only consists of the
  simple reflections corresponding to $\alpha_1,\dots,\alpha_s$ we
  must have that $\gamma_{N_s+1}=\alpha_{s+1}+\sum_{i=1}^s m_i
  \alpha_i$ for some coefficients $m_i\in \setN$. So
  $\{\beta_1,\dots,\beta_s,\gamma_{N_s+1}\}$ is a basis of
  $\Phi_{\{\alpha_1,\dots,\alpha_{s+1}\}}$. From Theorem~\ref{thm:DP}
  we get for $1\leq i \leq s$
  \begin{equation*}
    [F_{\gamma_{N_s+1}},F_{\beta_i}]_q \in \spa{\setC}{F_{\gamma_{2}}^{a_2}\cdots F_{\gamma_{N_s}}^{a_{N_s}}|a_i\in \setN\}}.
  \end{equation*}
  But since $\{\gamma_1,\dots,\gamma_{N_s}\} =
  \Phi^+_{\{\alpha_1,\dots,\alpha_s\}}$ and since
  $\gamma_{N_s+1}=\alpha_{s+1}+\sum_{i=1}^s m_i \alpha_i$ we get
  $[F_{\gamma_{N_s+1}},F_{\beta_i}]_q=0$.

  All that is left is to show that $\gamma_{N_s+1}\not \in \Phi_F$. By
  the above we must have that $\alpha_{s+1}$ is connected to some
  $\alpha_i\in J'$. We will show that the coefficient of $\alpha_i$ in
  $\gamma_{N_s+1}$ is nonzero. Otherwise $(\gamma_{N_s+1}|\alpha_i)<0$
  and so $\gamma_{N_s+1}+\alpha_i\in
  \Phi_{\{\alpha_1,\dots,\alpha_{s+1}\}}$ and by Theorem~1
  in~\cite{Papi}, $\gamma_{N_s+1}+\alpha_i = \gamma_{j}$ for some $1 <
  j \leq s$. This is impossible since $\gamma_{N_s+1}+\alpha_i\not \in
  \Phi_{\{\alpha_1,\dots,\alpha_{s}\}}$. So we can set $\beta_{s+1} =
  \gamma_{N_s+1}$ and the induction step is finished.

  To prove assertion $1.$ it can be assumed that $I=\Pi$. Thus
  assertion $1.$ follows from assertion $2.$ with $J=\{\alpha\}$ and
  $F=\emptyset$.
\end{proof}

\begin{lemma}
  \label{lemma:35}
  Let $\Sigma=\{\beta_1,\dots,\beta_n\}$ be a set of commuting roots
  with corresponding root vectors $F_{\beta_1},\dots,F_{\beta_n}$,
  then $F_{\beta_1}^{(l)},\dots,F_{\beta_n}^{(l)}$ commute.
\end{lemma}
\begin{proof}
  Calculating in $U_v$ for $i<j$ we get using
  Proposition~\ref{derivative}
  \begin{equation*}
    [F_{ \beta_j }^{(l)}, F_{ \beta_i }^{(l)}  ]_v =
    \frac{1}{([l]_v!)^2} [F_{\beta_j}^l,F_{\beta_i}^l]_v = 0
  \end{equation*}
  hence
  $v^{(l\beta_i|l\beta_j)}F_{\beta_i}^{(l)}F_{\beta_j}^{(l)}-F_{\beta_j}^{(l)}F_{\beta_i}^{(l)}=0$
  in $U_A$. Since $[F_{ \beta_i}^{(l)},F_{\beta_j}^{(l)}]_q =
  q^{l^2(\beta_i|\beta_j)}F_{\beta_i}^{(l)}F_{\beta_j}^{(l)}-F_{\beta_j}^{(l)}F_{\beta_i}^{(l)}
  =
  F_{\beta_i}^{(l)}F_{\beta_j}^{(l)}-F_{\beta_j}^{(l)}F_{\beta_i}^{(l)}
  $ we have proved the lemma.
\end{proof}

\begin{cor}
  \label{cor:7}
  Let $\Sigma=\{\beta_1,\dots,\beta_n\}$ be a set of commuting roots
  with corresponding root vectors $F_{\beta_1},\dots,F_{\beta_n}$. The
  set
  \begin{align*}
    F_\Sigma:=&\{r_1!F_{\beta_1}^{(r_1l)}\cdots
    r_n!F_{\beta_n}^{(r_nl)}|r_1,\dots,r_n\in \setN\}
    \\
    =&\{\left(F_{\beta_1}^{(l)}\right)^{r_1} \cdots
    \left(F_{\beta_n}^{(l)}\right)^{r_n}|r_1,\dots,r_n\in \setN\}
  \end{align*}
  is an Ore subset of $U_q$.
\end{cor}
\begin{proof}
  This follows from Lemma~\ref{lemma:35} and Lemma~\ref{lemma:19}.
\end{proof}

We let $U_{q(F_\Sigma)}$ denote the Ore localization of $U_q$ in the
Ore subset $F_\Sigma$. For a $U_q$-module $M$ we define $M_{F_\Sigma}
= U_{q(F_\Sigma)}\tensor_{U_q} M$.

\begin{defn}
  Let $\Sigma=\{\beta_1,\dots,\beta_n\}$ be a set of commuting roots
  that is a basis of $Q$ with a corresponding Ore subset $F_\Sigma$.
  Let $\nu \in \mathfrak{h}^*$, $\nu=\sum_{i=1}^n a_i \beta_i$ for
  some $a_i\in \setC$. For a $U_{q(F_\Sigma)}$-module $M$ we define
  $\psi_{F_\Sigma,\nu}.M = \psi_{F_{\beta_1},a_1}\circ \dots \circ
  \psi_{F_{\beta_n},a_n}.M$.
\end{defn}

\begin{cor}
  \label{cor:5}
  Let $\Sigma$ be a set of commuting roots that is a basis of $Q$. Let
  $\mu \in Q$ and let $M$ be a $U_{q(F_\Sigma)}$-module. Then
  \begin{equation*}
    \psi_{F_\Sigma,\mu}.M \iso M
  \end{equation*}
  as $U_{q(F_\Sigma)}$-modules. Also for $\lambda =
  (\lambda^0,\lambda^1)\in \wt M$:
  \begin{equation*}
    \psi_{F_\Sigma,\mu}.M_\lambda \iso M_{(\lambda^0,\lambda^1+\mu)}
  \end{equation*}
  as $(U_{q(F_\Sigma)})_0$-modules.
\end{cor}
\begin{proof}
  Since $\Sigma$ is a basis of $Q$ we can write $\mu = \sum_{\beta\in
    \Sigma} a_\beta \beta$ for some $a_\beta\in \setZ$. So the
  corollary follows from Lemma~\ref{lemma:31}.
\end{proof}

\begin{defn}
  A module $M\in \mathcal{F}$ is called admissible if its weights are
  contained in a single coset of $(\Lambda_l\times
  \mathfrak{h}^*)/(\Lambda_l \times Q)$ and if the dimensions of the
  weight spaces are uniformly bounded. $M$ is called admissible of
  degree $d$ if $d$ is the maximal dimension of the weight spaces in
  $M$.
\end{defn}

Of course all finite dimensional simple modules are admissible but the
interesting admissible modules are the infinite dimensional admissible
simple modules. In particular simple torsion free modules in
$\mathcal{F}$ are admissible. We show later that each infinite
dimensional simple module $L$ gives rise to a 'coherent family'
$\mathcal{EXT}(L)$ containing at least one simple highest weight
module that is admissible of the same degree.

We need the equivalent of Lemma~3.3 in~\cite{Mathieu}. Some of the
proofs leading up to this are more or less the same as
in~\cite{Mathieu} but for completenes we include it here as well.

\begin{defn}
  A cone $C$ is a finitely generated submonoid of the root lattice $Q$
  containing $0$.  If $L$ is a simple module define the cone of $L$,
  $C(L)$, to be the submonoid of $Q$ generated by $T_L$.
\end{defn}

\begin{lemma}
  \label{lemma:11}
  Let $L\in\mathcal{F}$ be an infinite dimensional simple module. Then
  the group generated by the submonoid $C(L)$ is $Q$.
\end{lemma}
Compare~\cite{Mathieu} Lemma~3.1
\begin{proof}
  First consider the case where $T_L \cap (-F_L) = \emptyset$. Then in
  this case we have $\Phi = T_L^s \cup F_L^s$. Since $F_L^s$ and
  $T_L^s$ contain different connected components of the Dynkin diagram
  and since $L$ is simple and infinite we must have $\Phi = T_L^s$ and
  therefore $C(L)=Q$.
  
  Next assume $T_L \cap (-F_L) \neq \emptyset$. By Lemma~4.16
  in~\cite{Fernando} $P_L=T_L^s\cup F_L$ and $P_L^-=T_L\cup F_L^s$ are
  two opposite parabolic subsystems of $\Phi$. So we have that
  $T_L\cap (-F_L)$ and $(-T_L)\cap F_L$ must be the roots
  corresponding to the nilradicals $\mathfrak{v}^\pm$ of two opposite
  parabolic subalgebras $\mathfrak{p}^\pm$ of $\mathfrak{g}$. Since we
  have $\mathfrak{g}=\mathfrak{v}^+ + \mathfrak{v}^- +
  [\mathfrak{v}^+,\mathfrak{v}^-]$ we get that $T_L\cap (-F_L)$
  generates $Q$. Since $C(L)$ contains $T_L\cap (-F_L)$ it generates
  $Q$.
\end{proof}

\begin{defn}
  Let $x\geq 0$ be a real number. Define $\rho_l(x) =
  \operatorname{Card}B_l(x)$ where $B_l(x)=\{\mu\in l Q|
  \sqrt{(\mu|\mu)} \leq x\}$ and $lQ = \{ l\mu \in Q| \mu\in Q\}$.

  Let $M\in \mathcal{F}$ be a weight module with all weights lying in
  a single coset of $(\Lambda_l\times \mathfrak{h}^*)/(\Lambda_l\times
  Q)$ say $(0,\lambda^1)+(\Lambda_l\times Q)$. The density of $M$ is
  \begin{equation*}
    \delta_l(M) =   \liminf_{x\to \infty} \rho_l(x)\inv \sum_{\mu^0\in \Lambda_l, \mu^1 \in
      B_l(x)} \dim M_{q^{\mu^1}(\mu^0,\lambda^1)}.
  \end{equation*}

  For a cone $C$ we define $\delta(C) = \liminf_{x\to \infty}
  \rho_1(x)\inv \operatorname{Card}(C \cap B_1(x))=\liminf_{x\to
    \infty} \rho_l(x)\inv \operatorname{Card}(l C \cap B_l(x))$ where
  $lC = \{ lc \in Q| c\in C\}$.
\end{defn}

\begin{lemma}
  \label{lemma:12}
  There exists a real number $\epsilon > 0$ such that $\delta_l
  (L)>\epsilon$ for all infinite dimensional simple modules $L$.
\end{lemma}
\begin{proof}
  Note that since $q^{l C(L)}\lambda \subset \wt L$ for all
  $\lambda\in \wt L$ we have $\delta_l(L) \geq \delta(C(L))$.

  Since $C(L)$ is the cone generated by $T_L$ and $T_L\subset \Phi$ (a
  finite set) there can only be finitely many different cones.

  Since there are only finitely many different cones attached to
  infinite simple dimensional modules and since any cone $C$ that
  generates $Q$ has $\delta(C)>0$ we conclude via Lemma~\ref{lemma:11}
  that there exists an $\epsilon>0$ such that $\delta_l(L)>\epsilon$
  for all infinite dimensional simple modules.
\end{proof}

\begin{defn}
  Let $M$ be a $\mathfrak{g}$-module. We can make $M$ into a
  $U_q$-module by the Frobenius homomorphism: We define $M^{[l]}$ to
  be the $U_q$-module equal to $M$ as a vector space and with the
  action defined as follows: For $m\in M$, $\alpha\in \Pi$,
  \begin{align*}
    K_\alpha^{\pm 1} m =& m
    \\
    E_\alpha m =& 0
    \\
    E_\alpha^{(l)} m=& e_\alpha m
    \\
    F_\alpha m =& 0
    \\
    F_\alpha^{(l)} m =& f_\alpha m
  \end{align*}
  where $e_\alpha$ is a root vector of $\mathfrak{g}$ of weight
  $\alpha$ and $f_\alpha$ is such that $[e_\alpha,f_\alpha]=h_\alpha$.
  The above defines an action of $U_q$ on $M$ by Theorem~1.1
  in~\cite{KL2}.
\end{defn}

\begin{prop}
  \label{prop:21}
  Let $\lambda=(\lambda^0,\lambda^1)\in X$ and let $L(\lambda)$ be the
  unique simple highest weight module with weight $\lambda$. Then
  $L(\lambda)\iso L_\setC (\lambda^1)^{[l]}\tensor L((\lambda^0,0))$
  where $L_\setC(\lambda^1)$ denotes the unique simple
  $\mathfrak{g}$-module of highest weight $\lambda^1$.
\end{prop}
\begin{proof}
  The proof of Theorem~3.1 in~\cite{catO} works here in exactly the
  same way.
\end{proof}

\begin{lemma}
  \label{lemma:10}
  Let $M\in \mathcal{F}$ be an admissible module. Then $M$ has finite
  Jordan-Hölder length.
\end{lemma}
\begin{proof}
  As $M$ is admissible, we have $\delta_l(M)<\infty$. For any exact
  sequence $0\to M_1 \to M_2 \to M_3 \to 0$, we have
  $\delta_l(M_2)\geq \delta_l(M_1)+\delta_l(M_3)$. Let $Y$ be the set
  of all $\mu\in \Lambda$ such that
  $|\left<\mu,\alpha^\vee\right>|\leq 1$ for all $\alpha\in \Pi$. By
  Proposition~\ref{prop:21} and the classification of classical simple
  finite dimensional $\mathfrak{g}$-modules any finite dimensional
  $U_q$-module $L$ has $L_{(\mu^0,\mu^1)}\neq 0$ for some $\mu^0\in
  \Lambda_l$ and some $\mu^1\in Y$. It follows like
  in~\cite[Lemma~3.3]{Mathieu} that the length of $M$ is finite and
  bounded by $A+\delta_l(M)/\epsilon$ where
  $A=\sum_{\mu^0\in\Lambda_l,\mu^1\in Y}\dim M_{(\mu^0,\mu^1)}$ and
  $\epsilon$ is the constant from Lemma~\ref{lemma:12}.
\end{proof}

\begin{lemma}
  \label{lemma:13}
  Let $M$ be an admissible module. Let $\Sigma\subset \Phi^+$ be a set
  of commuting roots and $F_\Sigma$ a corresponding Ore subset. Assume
  $-\Sigma\subset T_M$. Then for $\lambda=(\lambda^0,\lambda^1) \in
  X$:
  \begin{equation*}
    \dim (M_{F_\Sigma})_\lambda = \max_{\mu\in \setZ \Sigma}\{\dim M_{(\lambda^0,\lambda^1+\mu)}\}
  \end{equation*}
  and if $\dim M_\lambda = \max_{\mu\in \setZ \Sigma}\{\dim
  M_{(\lambda^0,\lambda^1+\mu)}\}$ then $(M_{F_\Sigma})_\lambda \iso
  M_\lambda$ as $(U_q)_0$-modules.

  In particular if $\Sigma\subset T_M$ as well then $M_{F_\Sigma}\iso
  M$ as $U_q$-modules.
\end{lemma}
Compare to Lemma~4.4(ii) in~\cite{Mathieu}.
\begin{proof}
  We have $\Sigma=\{\beta_1,\dots,\beta_r\}$ for some
  $\beta_1,\dots,\beta_r\in \Phi^+$. Let
  $F_{\beta_1},\dots,F_{\beta_r}$ be corresponding $q$-commuting root
  vectors. Let $\lambda\in X$ and set
  \begin{equation*}
    d=\max_{\mu\in \setZ
      \Sigma}\{\dim M_{(\lambda^0,\lambda^1+\mu)}\}.
  \end{equation*}
  Let $V$ be a finite dimensional subspace of
  $(M_{F_\Sigma})_\lambda$. Then there exists a homogenous element
  $s\in F_\Sigma$ such that $sV\subset M$. Let $\nu\in \setZ\Sigma$ be
  the degree of $s$. So $sV \subset M_{q^\nu\lambda}$ hence $\dim sV
  \leq d$. Since $s$ acts injectively on $M_{F_\Sigma}$ we have $\dim
  V \leq d$. Now the first claim follows because $F_\beta^{(\pm l)}$
  acts injectively on $M_{F_\Sigma}$ for all $\beta\in \Sigma$.
  
  We have an injective $U_q$-homomorphism from $M$ to $M_{F_\Sigma}$
  sending $m\in M$ to $1\tensor m\in M_{F_\Sigma}$ that restricts to a
  $(U_q)_0$-homomorphism from $M_\lambda$ to
  $(M_{F_\Sigma})_\lambda$. If $\dim M_\lambda = d$ then this is
  surjective as well. So it is an isomorphism. The last claim follows
  because $\pm \Sigma \subset T_M$ implies $\dim M_\lambda = \dim
  M_{q^\mu\lambda}$ for any $\mu\in \setZ l \Sigma$; so $M_\lambda
  \iso (M_{F_\Sigma})_\lambda$ for any $\lambda\in X$. Since $M$ is a
  weight module this implies that $M \iso M_{F_\Sigma}$ as
  $U_q$-modules.
\end{proof}

\begin{lemma}
  \label{lemma:21}
  Let $L\in\mathcal{F}$ be a simple $U_q(\mathfrak{sl}_2)$
  module. Then the weight spaces of $L$ are all $1$-dimensional.
\end{lemma}
\begin{proof}
  For $\mathfrak{sl}_2$ there is only one simple root $\alpha$ and we
  will denote the root vectors $E_{\alpha}$ and $F_{\alpha}$ by $E$
  and $F$ respectively. Similarly $K^{\pm 1} = K_\alpha^{\pm
    1}$. Consider the Casimir element $C= EF + \frac{q\inv K + q
    K\inv}{(q-q\inv)^2}$. Let $\lambda\in\wt L$ and let $c\in\setC$ be
  an eigenvalue of $C$ on $L_\lambda$. Consider the eigenspace $L(c) =
  \{v\in L_\lambda | C v = cv \}$. Then $F^{(l)}E^{(l)}$ acts on this
  space since $C$ commutes with all elements from
  $U_q(\mathfrak{sl}_2)$. Choose an eigenvector $v_0\in L(c)$ for
  $F^{(l)}E^{(l)}$. We will show by induction that
  $E^{(n)}F^{(n)}v_0\in \setC v_0$ for all $n\in\setN$. The induction
  start $n=0$ is obivous. Let $n\in\setN$ and assume $n=i+rl$ with
  $0\leq i< l$. If $i\neq 0$ then $[n] \neq 0$ and we have:
  \begin{align*}
    E^{(n)}F^{(n)}v_0 =& \frac{1}{[n]^2} E^{(n-1)} E F F^{(n-1)}v_0
    \\
    = &\frac{1}{[n]^2} E^{(n-1)}\left(C-\frac{q\inv K + q
        K\inv}{(q-q\inv)^2}\right)F^{(n-1)}v_0
    \\
    = & \frac{1}{[n]^2}
    E^{(n-1)}F^{(n-1)}\left(c-\frac{q^{\left(\lambda^0|\alpha
          \right)+1-2n} + q^{2n-1-\left(\lambda^0|\alpha
          \right)}}{(q-q\inv)^2}\right)v_0
  \end{align*}
  where $\alpha$ is the simple root. So the claim follows by
  induction. In the case that $i=0$ we have
  \begin{align*}
    E^{(n)}F^{(n)}v_0 =& \frac{1}{r} E^{(rl-l)}E^{(l)}F^{(rl)}v_0
    \\
    =& \frac{1}{r} E^{(rl-l)}\sum_{t=0}^l F^{(rl-t)}{K; 2t-rl-l\brack
      t} E^{(l-t)} v_0
    \\
    =& \frac{1}{r} E^{(rl-l)}\sum_{t=0}^l F^{(rl-l)}F^{(l-t)}
    E^{(l-t)}{K; l-rl\brack t} v_0
    \\
    =& \frac{1}{r} E^{(rl-l)}F^{(rl-l)}\left(F^{(l)}E^{(l)} +
      \sum_{t=1}^{l-1} F^{(l-t)} E^{(l-t)}{\left(\lambda^0|\alpha
        \right)+ l-rl\brack t} + \left<\lambda^1,\alpha^\vee\right>
      +1-r \right)v_0
  \end{align*}
  Since $v_0$ is an eigenvector for $F^{(l)}E^{(l)}$ we have only left
  to consider the action of $F^{(i)}E^{(i)}$ for $1\leq i< l$. But we
  can show like above that $F^{(i)}E^{(i)}v_0 \in \setC v_0$ by using
  that $C=FE + \frac{q K + K\inv q\inv}{(q-q\inv)^2}$.

  Now since $L$ is simple we must have that $L_\lambda$ is a simple
  $(U_q)_0$-module (Lemma~\ref{lemma:1}). So $L_\lambda$ is generated
  by $v_0$. Since $E^{(n)}F^{(n)}v_0\in \setC v_0$ for all $n\in\setC$
  we get $\dim L_\lambda = 1$.
\end{proof}

\begin{lemma}
  \label{lemma:24}
  Let $L$ be a simple infinite dimensional admissible module. Let
  $\beta\in (T_L^s)^+$. Then there exists a $b\in \setC$ such that
  $\psi_{F_\beta,b}.L_{F_\beta}$ contains a simple admissible
  $U_q$-submodule $L'$ with $T_{L'}\subset T_L$ and $\beta\not \in
  T_{L'}$.
\end{lemma}
\begin{proof}
  By Lemma~\ref{lemma:13} $L\iso L_{F_\beta}$ as $U_q$-modules so we
  will write $L$ instead of $L_{F_\beta}$ when taking twist. The
  $U_{q(F_\beta)}$-module structure on $L$ coming from the
  isomorphism. Let $D_\beta$ be the subalgebra of $U_q$ generated by
  $E_{\beta}^{(n)}$, $K_{\beta}^{\pm 1}$, $F_{\beta}^{(n)}$, $n\in
  \setN$. Then $D_{\beta}$ is isomorphic to the algebra
  $U_{q_\beta}(\mathfrak{sl}_2)$. Let $v\in L$ and consider the
  $D_\beta$-module $D_\beta v$. Since $L$ is admissible so is $D_\beta
  v$. So $D_\beta v$ has a simple $D_\beta$-submodule $V$ by
  Lemma~\ref{lemma:10}.

  Let $v\in V$ be a weight vector such that $E_\beta v=0$ (such a $v$
  always exists since $E_\beta^l=0$). Assume $\lambda$ is the weight
  of $v$. By Lemma~\ref{lemma:21} $F_{\beta}^{(l)}E_{\beta}^{(l)}v = c
  v$ for some $c\in \setC$.
  
  Then by (the proof of) Lemma~\ref{lemma:17} we get
  \begin{align*}
    F_{\beta}^{(l)}E_{\beta}^{(l)}& \psi_{F_\beta,b}.v
    \\
    =& \psi_{F_\beta,b}.\left( c + b \sum_{t=1}^{l-1} F_\beta^{(l-t)}
      E_\beta^{(l-t)} {\left<\lambda_0,\beta^\vee\right> \brack
        t}_\beta + b (\left<\lambda_1,\beta^\vee\right> +1- b)\right)v
    \\
    =& \psi_{F_\beta,b}.\left( c + b
      (\left<\lambda_1,\beta^\vee\right> +1- b)\right)v.
  \end{align*}
  
  Since $\setC$ is algebraically closed the polynomial in $b$, $ c + b
  (\left<\lambda_1,\beta^\vee\right> +1- b)$ has a zero. Assume from
  now on that $b\in \setC$ is chosen such that $c +
  b(\left<\lambda_1,\beta^\vee\right> +1- b)=0$.

  Thus $\psi_{F_{\beta} ,b}.L$ contains an element
  $v'=\psi_{F_\beta,b}.v$ such that $F_\beta^{(l)} E_\beta^{(l)} v'=0$
  and since $F_\beta^{(l)}$ acts injectively on $\psi_{F_{\beta}
    ,b}.L$, we have $E_\beta^{(l)} v'=0$. Set $V=\{m\in
  \psi_{F_{\beta} ,b}.L| E_\beta^{(N)} m = 0, N>>0\}=(\psi_{F_{\beta}
    ,b}.L)^{[\beta]}$. By Proposition~\ref{prop:2} this is a
  $U_q$-submodule of the $U_q$-module $\psi_{F_{\beta} ,b}.L$. It is
  nonzero since $v'\in V$. By Lemma~\ref{lemma:10} $V$ has a simple
  $U_q$-submodule $L'$.

  We have left to show that $T_{L'} \subset T_L$. Assume $\gamma \in
  T_{L'}$. Then $q^{l \setN \gamma}\wt L' \subset \wt L'$ by
  Proposition~\ref{prop:13}. But since $\wt L' \subset \{(\lambda^0,\lambda^1-b\beta)| (\lambda^0,\lambda^1)\in \wt L\}$ we get for some $\nu \in \wt L$, $ \{(\nu^0,\nu^1-b\beta + r\gamma)|r\in \setN\} \subset \{(\lambda^0,\lambda^1-b\beta)| (\lambda^0,\lambda^1)\in \wt L\}$ or equivalently $q^{l \setN
    \gamma} \nu \subset \wt L$. But this shows that $\gamma \not \in
  F_L$ and since $L$ is a simple $U_q$-module this implies that
  $\gamma \in T_L$. By construction we have $\beta\not \in T_{L'}$.
\end{proof}

\begin{lemma}
  \label{lemma:6}
  Let $L\in \mathcal{F}$ be a simple module. Then there exists a $w\in
  W$ such that $w(F_L\backslash F_L^s) \subset \Phi^+$ and
  $w(T_L\backslash T_L^s) \subset \Phi^-$.
\end{lemma}
\begin{proof}
  Lemma~4.16 in~\cite{Fernando} tells us that there exists a basis $B$
  of the root system $\Phi$ such that the antisymmetrical part,
  $F_L\backslash F_L^s$, of $F_L$ is contained in the positive roots
  $\Phi_B^+$ corresponding to the basis $B$ and the antisymmetrical
  part, $T_L\backslash T_L^s$, of $T_L$ is contained in the negative
  roots $\Phi_B^-$ corresponding to the basis. Since all bases of a
  root system are $W$-conjugate the claim follows.
\end{proof}

\begin{lemma}
  \label{lemma:26}
  Let $L$ be an infinite dimensional admissible simple module. Let
  $w\in W$ be such that $w(F_L\backslash F_L^s) \subset \Phi^+$. Let
  $\alpha\in \Pi$ be such that $-\alpha\in w(T_L)$ (such an $\alpha$
  always exists). Then there exists a commuting set of roots $\Sigma$
  with $\alpha\in \Sigma$ which is a basis of $Q$ such that $ -\Sigma
  \subset w(T_L)$.
\end{lemma}
\begin{proof}
  Set $L' = {^w}L$. Since $w(T_L) = T_{{^{w}}L}=T_{L'}$ we will just
  work with $L'$. Then $F_{L'}\backslash F_{L'}^s \subset \Phi^+$.
  
  Note that it is always possible to choose a simple root $\alpha\in
  -T_{L'}$ since $L'$ is infinite dimensional: If this was not
  possible we would have $\Phi^- \subset F_{L'}$. But since
  $F_{L'}\backslash F_{L'}^s \subset \Phi^+$ this implies $F_{L'} =
  \Phi$.

  Set $F = F_{L'}^s \cap \Pi$. Since $L'$ is infinite dimensional
  $F\neq \Pi$.  By Lemma~\ref{lemma:7} $2.$ applied with
  $J=\{\alpha\}=\Sigma'$ there exists a commuting set of roots
  $\Sigma$ that is a basis of $Q$ such that $\Sigma \subset \Phi^+
  \backslash \Phi^+_F$. Since $F_{L'}\backslash F_{L'}^s \subset
  \Phi^+$ we have $\Phi^- = T_{L'}^- \cup (F_{L'}^s)^-$.  To show
  $-\Sigma \subset T_{L'}$ we show $\left( \Phi^- \backslash \Phi^-_F
  \right)\cap F_{L'}^s=\emptyset$ or equivalently $(F_{L'}^s)^-
  \subset \Phi_F^-$.

  Assume $\beta\in F_{L'}^s\cap \Phi^+$, $\beta = \sum_{\alpha\in \Pi}
  a_\alpha \alpha$, $a_\alpha\in\setN$. The height of $\beta$ is the
  sum $\sum_{\alpha\in \Pi} a_\alpha$. We will show by induction on
  the height of $\beta$ that $-\beta\in \Phi_F^-$. If the height of
  $\beta$ is $1$ then $\beta$ is a simple root and so $\beta\in
  F$. Clearly $-\beta\in \Phi_F^-$ in this case. Assume the height of
  $\beta$ is greater than $1$. Let $\alpha'\in \Pi$ be a simple root
  such that $\beta-\alpha'$ is a root. There are two possibilities:
  $-\alpha' \in T_{L'}$ or $\pm \alpha' \in F_{L'}^s$.
  
  In the first case where $-\alpha'\in T_{L'}$ we must have $-\beta +
  \alpha' \in F_{L'}^s$ since if $-\beta + \alpha' \in T_L$ then
  $-\beta = (-\beta + \alpha') - \alpha' \in T_{L'}$ because $T_{L'}$
  is closed (Proposition~\ref{prop:14}). So $\beta-\alpha' \in
  F_{L'}^s$ and $\beta \in F_{L'}^s$. Since $F_{L'}$ is closed
  (Proposition~\ref{prop:5}) we get $-\alpha' = (\beta-\alpha') -
  \beta \in F_L$ which is a contradiction. So the first case
  ($-\alpha' \in T_{L'}$) is impossible.
  
  In the second case since $F_{L'}$ is closed we get $\pm (\beta -
  \alpha') \in F_{L'}$ i.e. $\beta- \alpha' \in F_{L'}^s$. By the
  induction $-(\beta-\alpha') \in \Phi_F^-$ and since $-\beta =
  -(\beta-\alpha') - \alpha'$ we are done.
\end{proof}

\section{Coherent families}
\label{sec:coherent-families}
As in the above section $q$ is a complex primitive $l$'th root of
unity with $l$ odd in this section. For $\lambda\in X$ we write
$\lambda = (\lambda^0,\lambda^1)$ like above.

\begin{lemma}
  \label{lemma:32}
  Let $M,N\in \mathcal{F}$ be semisimple $U_q$-modules. If $\Tr^M =
  \Tr^N$ then $M\iso N$.
\end{lemma}
\begin{proof}
  Theorem~7.19 in~\cite{Lam} states that this is true for modules over
  a \emph{finite dimensional} algebra. So we will reduce to the case
  of modules over a finite dimensional algebra. Let $L$ be a
  composition factor of $M$ and $\lambda$ a weight of $L$. Then the
  multiplicity of the $U_q$-composition factor $L$ in $M$ is the
  multiplicity of the $(U_q)_0$-composition factor $L_\lambda$ in
  $M_\lambda$ by Theorem~\ref{thm:Lemire}. $M_\lambda$ is a finite
  dimensional $(U_q)_0$-module. Let $I$ be the kernel of the
  homomorphism $(U_q)_0 \to \End_{\setC}(M_\lambda )$ given by the
  action of $(U_q)_0$. Then $(U_q)_0 / I$ is a finite dimensional
  $\setC$ algebra and $M_\lambda$ is a module over $(U_q)_0/
  I$. Furthermore since $\Tr^M(\lambda,u)=0$ for all $u\in I$ the
  trace of an element $u\in (U_q)_0$ is the same as the trace of $u+I
  \in (U_q)_0/I$ on $M_\lambda$ as a $(U_q)_0/I$-module. So if $\Tr^M
  = \Tr^N$ the multiplicity of $L_\lambda$ in $M_\lambda$ and
  $N_\lambda$ are the same and hence the multiplicity of $L$ in $M$ is
  the same as in $N$.
\end{proof}

\begin{defn}
  \begin{equation*}
    T^* = \mathfrak{h}^* / Q.
  \end{equation*}
\end{defn}

By Corollary~\ref{cor:5} it makes sense to write $\psi_{F_\beta,t}.M$
for $t\in T^*$ up to isomorphism for a $U_{q(F_\Sigma)}$-module $M$.

\begin{defn}
  \label{sec:root-unity-stuff}
  A (quantized) coherent family is a $U_q$-module $\mathcal{M}$ such
  that for all $\mu \in \Lambda_l$:
  \begin{itemize}
  \item $\dim \mathcal{M}_{(\mu,\nu)} = \dim \mathcal{M}_{(\mu , \nu')
    }$ for all $\nu,\nu' \in \mathfrak{h}^*$.
  \item For all $u\in (U_q)_0$, the map
    $\mathfrak{h}^* \ni \nu \mapsto \Tr u|_{\mathcal{M}_{(\mu ,\nu)}}$
    is polynomial.
  \end{itemize}
  
  For a coherent family $\mathcal{M}$ and $t\in T^*$ define
  \begin{align*}
    \mathcal{M}[t] = \bigoplus_{\mu^0\in \Lambda_l, \mu^1 \in t}
    \mathcal{M}_{(\mu^0,\mu^1)}.
  \end{align*}
  
  $\mathcal{M}$ is called irreducible if there exists a $t\in T^*$
  such that $\mathcal{M}[t]$ is a simple $U_q$-module.
\end{defn}

\begin{lemma}
  \label{lemma:8}
  Let $\mathcal{M}$ be a coherent family. Let $\mu \in
  \Lambda_l$. Then the set $\Omega$ of all weights $\nu\in
  \mathfrak{h}^*$ such that the $(U_q)_0$-module
  $\mathcal{M}_{(\mu,\nu)}$ is simple is a Zariski open subset of
  $\mathfrak{h}^*$. 

  If $\mathcal{M}$ is irreducible then $\Omega\neq \emptyset$ if
  $\mathcal{M}_{(\mu,\nu)}\neq 0$ for any $\nu\in \mathfrak{h}^*$
  (equivalently for all $\nu\in \mathfrak{h}^*$).
\end{lemma}
\begin{proof}
  If $\mathcal{M}_{(\mu,\nu)}=0$ for all $\nu\in \mathfrak{h}^*$ then
  $\Omega=\emptyset$. Assume $\dim \mathcal{M}_{(\mu,\nu)} = d>0$ for
  all $\nu \in \mathfrak{h}^*$.  If $\mathcal{M}$ is irreducible
  there exists $t\in T^*$ such that $\mathcal{M}[t]$ is a simple
  $U_q$-module. Then for $\nu\in t$, $\mathcal{M}_{(\mu,\nu)}=
  \mathcal{M}[t]_{(\mu,\nu)}$ is a simple $U_q$-module by
  Theorem~\ref{thm:Lemire}. So in this case $\Omega\neq \emptyset$.

  Now the proof goes exactly like in~\cite[Lemma~4.7]{Mathieu}: The
  $(U_q)_0$-module $\mathcal{M}_{(\mu,\nu)}$ is simple if and only if
  the bilinear map $B_{\nu}:(U_q)_0\times (U_q)_0 \ni (u,v) \mapsto
  \Tr (uv|_{\mathcal{M}_{(\mu,\nu)}})$ has maximal rank $d^2$. For any
  finite dimensional subspace $E\subset (U_q)_0$ the set $\Omega_E$ of
  all $\nu$ such that $B_{\nu}|_{E}$ has rank $d^2$ is open. Therefore
  $\Omega = \cup_{E} \Omega_E$ is open.
\end{proof}

\begin{defn}
  Let $L$ be an admissible $U_q$-module and let $\mu\in \Lambda_l$.
  \begin{equation*}
    \Supp (L,\mu) = \{ \nu \in \mathfrak{h}^* | \dim L_{(\mu, \nu)}>0 \}
  \end{equation*}
  and
  \begin{equation*}
    \Suppess(L,\mu) = \{ \nu \in \Supp(L,\mu) | \dim L_{(\mu, \nu)} \text{ is maximal in } \{\dim L_{(\mu, \nu')}| \nu'\in\mathfrak{h}^*\} \}.
  \end{equation*}
\end{defn}



\begin{defn}
  Let $M$ be an admissible module. Define $M^{ss}$ to be the unique
  (up to isomorphism) semisimple module with the same composition
  factors as $M$.

  Let $V$ be a $U_q$-module such that $V=\bigoplus_{i\in I} V_i$ for
  some index set $I$ and some admissible $U_q$-modules $V_i$. Then
  $V^{ss} = \bigoplus_{i\in I} V_i^{ss}$.
\end{defn}

  

  

\begin{prop}
  \label{prop:20}
  Let $L$ be an infinite dimensional admissible simple
  $U_q$-module. Then there exists a unique (up to isomorphism)
  semisimple irreducible coherent family $\mathcal{EXT}(L)$ containing
  $L$.
\end{prop}
\begin{proof}
  Let $w\in W$ be such that $w(F_L\backslash F_L^s)\subset \Phi^+$ and
  $\Sigma$ a set of commuting roots that is a basis of $Q$ such that
  $-\Sigma\subset w(T_L)$ (Exists by Lemma~\ref{lemma:6} and
  Lemma~\ref{lemma:26}) with corresponding Ore subset $F_\Sigma$. Set
  \begin{equation*}
    \mathcal{EXT}(L):= \left(\bigoplus_{t \in T^*} {^{\bar{w}}}\left(
        \psi_{F_\Sigma,t} . ({^w}L)_{F_\Sigma} \right)
    \right)^{ss}.
  \end{equation*}
  For each $t\in T^*$ choose a representative $\nu_t\in t$. As a
  $(U_{q(F_\Sigma)})_0$-module
  \begin{equation*}
    \mathcal{EXT}(L)= \bigoplus_{t \in T^*}
    {^{\bar{w}}}\left( \psi_{F_\Sigma,\nu_t} . ({^w}L)_{F_\Sigma} \right)^{ss}.
  \end{equation*}
  Define $Y := \{\mu\in \Lambda_l | \Supp({^w}L,\mu)\neq
  \emptyset\}$. For each $\mu\in Y$ let $\lambda_\mu \in
  \Suppess({^w}L,\mu)$. By Corollary~\ref{cor:5}
  \begin{equation*}
    ({^w}L)_{F_\Sigma} \iso \bigoplus_{\mu \in Y} \bigoplus_{\nu\in Q} \psi_{F_\Sigma,\nu}.(({^w}L)_{F_\Sigma})_{(\mu,\lambda_\mu)}
  \end{equation*}
  as $(U_{q(F_\Sigma)})_0$-modules.
  
  So we have the following $(U_{q(F_\Sigma)})_0$-module isomorphisms:
  \begin{align*}
    \mathcal{EXT}(L)\iso& \bigoplus_{\mu\in Y} \bigoplus_{t \in T^*}
    \bigoplus_{\nu \in Q} {^{\bar{w}}}\left( \psi_{F_\Sigma,\nu_t+\nu}
      . (({^w}L)_{F_\Sigma})_{(\mu,\lambda_\mu)} \right)^{ss}
    \\
    \iso& \bigoplus_{\mu\in Y} \bigoplus_{\nu \in \mathfrak{h}^* }
    {^{\bar{w}}}\left( \psi_{F_\Sigma, \nu }
      . (({^w}L)_{F_\Sigma})_{(\mu,\lambda_\mu)} \right)^{ss}.
  \end{align*}
  
  Let $u\in (U_q)_0$ and $\mu\in Y$. Then we see from the above and
  Lemma~\ref{lemma:13} that
  \begin{equation*}
    \Tr u|_{\mathcal{EXT}(L)_{(\mu,\nu)}} = \Tr \psi_{F_\Sigma,\nu-\lambda_\mu}^{(\mu,\lambda_\mu)}(T_w\inv (u))|_{ ({^w} L)_{(\mu,\lambda_\mu)}}.
  \end{equation*}
  By Lemma~\ref{lemma:17} this is polynomial in $\nu-\lambda_\mu$
  hence also polynomial in $\nu$.  We know that this polynomial is
  determined in all $\nu$ such that $\nu-\lambda_\mu \in
  \Suppess(L,\mu)$. $\Suppess(L,\mu)$ is Zariski dense in
  $\mathfrak{h}^*$ because $\lambda_\mu -\setN \Sigma \subset
  \Suppess(L,\mu)$ and $\Sigma$ is a basis of $Q$. So $\Tr$ is
  determined on all of $\mathcal{EXT}(L)$ by $L$. For any
  $(\mu,\nu)\in X$ we have
  \begin{align*}
    \dim \mathcal{EXT}(L)_{(\mu,\lambda_\mu+\nu)}=&
    \dim\left(\psi_{F_\Sigma,\nu}.(({^w}L)_{F_\Sigma})_{(\mu,\lambda_\mu)}\right)^{ss}
    \\
    =& \dim (({^w}L)_{F_\Sigma})_{(\mu,\lambda_\mu)}
  \end{align*}
  so $\mathcal{EXT}(L)$ is a coherent family.
  
  Assume $\mathcal{M}$ is a semisimple irreducible coherent family
  containing $L$.  Let $\mu\in Y$. By Lemma~\ref{lemma:8} the set
  $\Omega_1$ of $\nu\in\mathfrak{h}^*$ such that
  $\mathcal{EXT}(L)_{(\mu,\nu)}$ is simple and the set $\Omega_2$ of
  $\nu\in\mathfrak{h}^*$ such that $\mathcal{M}_{(\mu,\nu)}$ is simple
  are non-empty open subsets of $\mathfrak{h}^*$ ($\Omega_1\neq
  \emptyset$ because $\mathcal{EXT}(L)_{(\mu,\nu)}= L_{(\mu,\nu)}$ for
  $\nu\in \Suppess(L,\mu)$). So their intersection $\Omega_1\cap
  \Omega_2$ is open and non-empty (since any Zariski open set of
  $\mathfrak{h}^*$ is Zariski dense in $\mathfrak{h}^*$). Since
  $\Suppess(L,\mu)$ is Zariski dense we get that there exists a
  $\nu\in \Omega_1\cap \Omega_2 \cap \Suppess(L,\mu)$ such that
  $\mathcal{M}_{(\mu,\nu)}$ and $\mathcal{EXT}(L)_{(\mu,\nu)}$ are
  simple. Since $L_{(\mu,\nu)}\subset \mathcal{M}_{(\mu,\nu)}$ and
  $L_{(\mu,\nu)}\subset \mathcal{EXT}(L)_{(\mu,\nu)}$ we get
  $\mathcal{M}_{(\mu,\nu)} \iso L_{(\mu,\nu)} \iso
  \mathcal{EXT}(L)_{(\mu,\nu)}$. This is true for any $(\mu,\nu)$ such
  that $\nu\in \Suppess(L,\mu)$. Let $u\in (U_q)_0$ and $\mu\in
  Y$. Then we see that $\Tr u|_{\mathcal{EXT}(L)_{(\mu,\nu)}}=\Tr
  u|_{L_{(\mu,\nu)}}=\Tr u|_{\mathcal{M}_{(\mu,\nu)}}$ for any $\nu\in
  \Suppess(L,\mu)$. Since $\Suppess(L,\mu)$ is Zariski dense this
  implies $\Tr u|_{\mathcal{EXT}(L)_{(\mu,\nu)}}=\Tr
  u|_{\mathcal{M}_{(\mu,\nu)}}$ for all $\nu\in \mathfrak{h}^*$. So by
  Lemma~\ref{lemma:32} $\mathcal{EXT}(L)_{(\mu,\nu)}\iso
  \mathcal{M}_{(\mu,\nu)}$ as $(U_q)_0$-modules for any $(\mu,\nu)\in
  \wt \mathcal{EXT}(L)$.

  Then by Theorem~\ref{thm:Lemire} we get that $\mathcal{M} \iso
  \mathcal{EXT}(L) \oplus \mathcal{N}$ for some coherent family
  $\mathcal{N}$ with the property that $\mathcal{N}_{(\mu,\nu)}=0$ for
  any $(\mu,\nu)\in X$ such that $\Supp(L,\mu)\neq \emptyset$. Since
  $\mathcal{M}$ is irreducible there exists a $t\in T^*$ such that the
  $U_q$-module $\mathcal{M}[t]$ is simple. We have $\mathcal{M}[t]
  \iso \mathcal{EXT}(L)[t] \oplus \mathcal{N}[t]$. Since
  $\mathcal{M}[t]$ is simple and $\mathcal{EXT}(L)[t]\neq 0$ we get
  that $\mathcal{N}[t]=0$. Since $\mathcal{N}$ is a coherent family
  this implies that $\mathcal{N}=0$. So $\mathcal{M} \iso
  \mathcal{EXT}(L)$.

  So we have left to show that $\mathcal{EXT}(L)$ is irreducible. Let
  $F_{\beta_1},\dots,F_{\beta_n}$ be the root vectors corresponding to
  $\Sigma=\{\beta_1,\dots,\beta_n\}$ and
  $E_{\beta_1},\dots,E_{\beta_n}$ the corresponding $E$-root
  vectors. Let $\mu \in Y$. As above we choose a $\lambda_\mu\in
  \Suppess({^w}L)$. The elements $F_{\beta_i}^{(l)}E_{\beta_i}^{(l)}$,
  $i=1,\dots,n$ act on
  $\psi_{F_\Sigma,\nu}.(({^w}L)_{F_\Sigma})_{(\mu,\lambda_\mu)}$ by
  $\sum_{j=1}^s p_{i,j}^\mu (\nu) u_{i,j}$ for some $u_{i,j}^\mu \in
  U_{q(F_\Sigma)}$ and some polynomials $p_{i,j}^\mu:\mathfrak{h}^*\to
  \setC$ so
  \begin{align*}
    p_\mu := \prod_{i=1}^n \det F_{\beta_i}^{(l)}E_{\beta_i}^{(l)}
    |_{\psi_{F_\Sigma,\nu}.(({^w}L)_{F_\Sigma})_{(\mu,\lambda_\mu)}}
  \end{align*}
  is a nonzero polynomial in $\nu$ by (the proof of)
  Lemma~\ref{lemma:17}. Set $p = \prod_{\mu\in Y} p_\mu$. Let $\Omega$
  be the set of non-zero points for
  $p$. By~\cite[Lemma~5.2~i)]{Mathieu} the set $T(\Omega):=
  \bigcap_{\mu\in Q} (\mu+\Omega)$ is non-empty. So there exists a
  $\nu\in \mathfrak{h}^*$ such that $p(\nu+\mu^1)\neq 0$ for any
  $\mu^1\in Q$. For such a $\nu$ we see that
  $F_{\beta_i}^{(l)}E_{\beta_i}^{(l)}$ act bijectively on
  \begin{align*}
    \bigoplus_{\mu \in Y}\bigoplus_{\mu^1\in Q}
    \psi_{F_\Sigma,\nu}.(({^w}L)_{F_\Sigma})_{(\mu,\lambda_\mu+\mu^1)}
    =& \psi_{F_{\Sigma},\nu}.({^w}L)_{F_\Sigma}.
  \end{align*}
  Since $F_{\beta_i}^{(l)}$ act injectively on
  $\psi_{F_{\Sigma},\nu}.({^w}L)_{F_\Sigma}$ this implies that
  $E_{\beta_i}^{(l)}$ act injectively on
  $\psi_{F_{\Sigma},\nu}.({^w}L)_{F_\Sigma}$. 
  Let $L_1\subset \psi_{F_{\Sigma},\nu}.({^w}L)_{F_\Sigma}$ be a
  simple $U_q$-submodule of
  $\psi_{F_{\Sigma},\nu}.({^w}L)_{F_\Sigma}$. By the above we have
  $\pm \Sigma \subset T_{L_1}$. So by Proposition~\ref{prop:13} we get
  $T_{L_1} = \Phi$. Define $\mathcal{EXT}(L_1) = \left(\bigoplus_{t
      \in T^*} \left( \psi_{F_\Sigma,t} . (L_1)_{F_\Sigma} \right)
  \right)^{ss}$. Then as above this is a coherent family. Let
  $\lambda'\in \wt L_1$. Then
  $\mathcal{EXT}[\lambda'+Q]=(L_1)_{F_\Sigma}=L_1$ by
  Lemma~\ref{lemma:13} so $\mathcal{EXT}(L_1)$ is an irreducible
  coherent family.

  Let $\mu \in \Lambda_l$ be such that $\Supp(L_1,\mu)\neq
  \emptyset$. $\Suppess(L_1,\mu)$ is Zariski dense in $\mathfrak{h}^*$
  so $\Suppess(L_1,\mu)\cap \Omega_1\neq \emptyset$. Let $\nu'\in
  \Omega_1 \cap \Suppess(L_1,\mu)$. Then $(L_1)_{(\mu,\nu')}\iso
  \left(
    \psi_{F_\Sigma,\nu}.({^w}L)_{F_\Sigma}\right)_{(\mu,\nu')}$. Then
  as above (with $\mathcal{M}=\mathcal{EXT}(L)$ and $L$ replaced by
  $L_1$) we get $\mathcal{EXT}(L) \iso \mathcal{EXT}(L_1) \oplus
  \mathcal{N}$ for some semisimple coherent family $\mathcal{N}$ with
  $\mathcal{N}_{(\mu,\nu)}=0$ for any $(\mu,\nu)\in X$ such that
  $\Supp(L_1,\mu)\neq \emptyset$. Since $\mathcal{EXT}(L)$ contains
  $L$ we get that $L = M' \oplus M''$ for some $U_q$-modules $M'
  \subset \mathcal{EXT}(L_1)$ and $M''\subset \mathcal{N}$. Since $L$
  is simple and since there exists a $\mu \in \Lambda_l$ such that
  $\Supp(L,\mu)\neq \emptyset$ and $\Supp(L_1,\mu)\neq \emptyset$ we
  must have $M''=0$ and $L=M'$. But then we have proved that the
  irreducible coherent family $\mathcal{EXT}(L_1)$ contains $L$. Hence
  $\mathcal{EXT}(L)\iso \mathcal{EXT}(L_1)$ by the above and
  $\mathcal{EXT}(L)$ is irreducible.
\end{proof}

\begin{thm}
  \label{thm:contains_highest_weight}
  Let $L$ be an admissible infinite dimensional simple module. Then
  there exists a $w\in W$ and a $\lambda\in X$ such that
  ${^w}\mathcal{EXT}(L)$ contains an infinite dimensional simple
  highest weight module $L(\lambda)$ and ${^w}\mathcal{EXT}(L)\iso
  \mathcal{EXT}(L(\lambda))$.
\end{thm}
\begin{proof}
  Let $w\in W$ be such that $w(F_L\backslash F_L^s) \subset \Phi^+$
  and $w(T_L\backslash T_L^s)\subset \Phi^-$ and let $\Sigma$ be a set
  of commuting roots that is a basis of $Q$ such that $-\Sigma\subset
  w(T_L)$ (Exists by Lemma~\ref{lemma:6} and
  Lemma~\ref{lemma:26}). Let $F_{\Sigma}$ be a corresponding Ore
  subset. Then
  
  \begin{align*}
    \mathcal{EXT}(L) =& \left(\bigoplus_{t \in T^*} {^{\bar{w}}}\left(
        \psi_{F_\Sigma,t} . ({^w}L)_{F_\Sigma} \right) \right)^{ss}
  \end{align*}
  so
  \begin{align*}
    {^w}\mathcal{EXT}(L) =& \left(\bigoplus_{t \in T^*} \left(
        \psi_{F_\Sigma,t} . ({^w}L)_{F_\Sigma} \right) \right)^{ss} = \mathcal{EXT}({^w}L).
  \end{align*}

  Set $L'={^{w}}L$. We will show by induction on $|T_{L'}^+|$ that
  there exists a $\lambda\in X$ such that $L(\lambda)$ is infinite
  dimensional and $\mathcal{EXT}(L')\iso \mathcal{EXT}(L(\lambda))$:

  If $|T_{L'}^+|=0$ then $L'$ is itself an infinite dimensional
  highest weight module. Assume $|T_{L'}^+|>0$. Then $T_{L'}^+\cap
  \Pi\neq \emptyset$ because if this was not the case then
  $\Phi^+\subset F_{L'}$ since $F_{L'}$ is closed. But $\Phi^+ \subset
  F_{L'}$ implies $|T_{L'}^+|=0$.

  Let $\alpha \in T_{L'}^+\cap \Pi$. Then $\alpha \in T_{L'}^s$ since
  $T_{L'}\backslash T_{L'}^s \subset \Phi^-$. So $-\alpha \in
  T_{L'}$. Then by Lemma~\ref{lemma:24} there exists a $b\in \setC$
  such that $\psi_{F_\alpha,b}.L'_{F_\alpha}$ contains a simple
  $U_q$-submodule $L''$ with $T_{L''}\subset T_{L'}$ and $\alpha \not
  \in T_{L''}$.  By Lemma~\ref{lemma:26} there exists a set of
  commuting roots $\Sigma$ that is a basis of $Q$ such that $\alpha\in
  \Sigma$ and $-\Sigma\subset T_{L'}$. Then by the above there exists
  a $\nu = b\alpha$ such that $\psi_{F_\Sigma,\nu}.L'_{F_\Sigma}$
  contains a simple $U_q$-submodule $L''$ with $T_{L''}\subset T_{L'}$
  and $\alpha \not \in T_{L''}$. $L''$ is infinite dimensional since
  $-\Sigma\subset T_{L''}$ and $\mathcal{EXT}(L'') \iso
  \mathcal{EXT}(L')$ by Proposition~\ref{prop:20}.

  By induction there exists a $\lambda\in X$ such that $L(\lambda)$ is
  infinite dimensional and $\mathcal{EXT}(L'')\iso
  \mathcal{EXT}(L(\lambda))$.
\end{proof}

The twists we have defined for quantum group modules are analogues of
the twists that can be made of normal Lie algebra modules as described
in~\cite{Mathieu}. In the next proposition we will use these Lie
algebra module twists denoted by $f_\Sigma^\nu$ given a set of
commuting roots $\Sigma$ and a $\nu\in T^*$ (see Section~4
in~\cite{Mathieu}). For $\lambda^1\in \mathfrak{h}^*$ let
$L_{\setC}(\lambda^1)$ denote the simple highest weight Lie algebra
$\mathfrak{g}$-module with highest weight $\lambda^1$. Let
$e_\beta,f_\beta$ denote root vectors in $\mathfrak{g}$ such that
$[e_\beta,f_\beta]=h_\beta$.

\begin{prop}
  \label{prop:22}
  Let $\lambda^1 \in \mathfrak{h}^*$ be such that
  $L_{\setC}(\lambda^1)$ is admissible. Let $\Sigma$ be a set of
  commuting roots that is a basis of $Q$ with $f_\beta$ acting
  injectively on $L_{\setC}(\lambda^1)$ for each $\beta\in
  \Sigma$. Let $\lambda^0\in \Lambda_l$. Define $\mathcal{M} = \left(
    \bigoplus_{\nu \in T^*} f_\Sigma^\nu.L_\setC
    (\lambda^1)_{f_\Sigma} \right)^{[l]}\tensor
  L((\lambda^0,0))$. Then $\mathcal{M}$ is an irreducible coherent
  family containing the simple highest weight module
  $L((\lambda^0,\lambda^1))$. 
\end{prop}
\begin{proof}
  $\mathcal{M}$ contains $L((\lambda^0, \lambda^1))$ by
  Proposition~\ref{prop:21}.

  Set $\mathcal{M}_\setC = \bigoplus_{\nu\in T^*} f_\Sigma^\nu.L_\setC
  (\lambda^1)_{f_\Sigma}$. So $\mathcal{M} =
  (\mathcal{M}_\setC)^{[l]}\tensor L((\lambda^0,0))$. Let $\mu\in
  \Lambda_l$ and $u\in (U_q)_0$. We need to show that the map $\nu
  \mapsto \Tr u|_{\mathcal{M}_{(\mu,\nu)}}$ is polynomial.
  
  \begin{align*}
    \mathcal{M}_{(\mu,\nu)} =& \bigoplus_{\eta \in \Lambda} \left(
      (\mathcal{M}_\setC)^{[l]}\right)_{q^\eta (0, \nu)} \tensor
    L((\lambda^0,0))_{q^{-\eta}(\mu,0)}
    \\
    =& \bigoplus_{\eta \in l\Lambda} \left(
      (\mathcal{M}_\setC)_{\nu+\frac{\eta}{l}}\right)^{[l]} \tensor
    L((\lambda^0,0))_{q^{-\eta}(\mu,0)}
    \\
    =& \bigoplus_{\eta \in l\Lambda} \left(
      f_\Sigma^{\nu+\frac{\eta}{l}}.(\mathcal{M}_\setC)_{0}\right)^{[l]}
    \tensor L((\lambda^0,0))_{q^{-\eta}(\mu,0)}.
  \end{align*}

  The action on $\left( f_\Sigma^\nu. (\mathcal{M}_\setC)_0
  \right)^{[l]}\tensor L(\lambda^0)$ is just the action on $\left(
    (\mathcal{M}_\setC)_0 \right)^{[l]}\tensor L(\lambda^0)$ twisted
  with the automorphism $u' \mapsto f_\Sigma^{\nu} u' f_\Sigma^{-\nu}$
  on the first tensor factor where $u'=\operatorname{Fr} (u)$
  ($\operatorname{Fr}$ is the Frobenius twist defined
  in~\cite[Theorem~1.1]{KL2}). The map $u' \mapsto f_\Sigma^{\nu} u'
  f_\Sigma^{-\nu}$ is of the form $\sum_i p_i(\nu) u_i$ for some
  polynomials $p_i$ and some $u_i\in (U_\setC)_0$ where
  $U_\setC:=U(\mathfrak{g})$ is the classical universal enveloping
  algebra of $\mathfrak{g}$. Composing a polynomial map with the map
  $\lambda\mapsto \lambda + \frac{\eta}{l}$ is still polynomial. So
  the trace is a finite sum of polynomials in $\lambda$ which is still
  polynomial.

  Let $u_q$ be the small quantum group as defined
  in~\cite{catO} i.e. the subalgebra of $U_q$
  generated by $E_{\alpha},K_{\alpha}^{\pm 1},F_{\alpha}$,
  $\alpha\in\Pi$. Then $L((\lambda^0,0))$ restricted to $u_q$ is a
  simple $u_q$-module by~\cite[Section~3.2]{catO}.
  
  By \cite[Lemma~5.3~i)]{Mathieu} and \cite[Proposition~5.4]{Mathieu}
  there exists a $t\in T^*$ such that $\mathcal{M}_{\setC}[t]$ is
  simple. Then $\mathcal{M}[t] = \left(
    \mathcal{M}_{\setC}[t]\right)^{[l]}\tensor L((\lambda^0,0))$ is
  simple: Let $0\neq
  v_0\tensor v_1\in L((\lambda^0,0)) \tensor \left(
    \mathcal{M}_{\setC}[t]\right)^{[l]}$.
  Then
  \begin{align*}
    U_q (v_0\tensor v_1) =& U_q u_q (v_0\tensor v_1)
    \\
    =& U_q (L((\lambda^0,0))\tensor v_1)
    \\
    =& L((\lambda^0,0)) \tensor U_q v_1
    \\
    =& L((\lambda^0,0)) \tensor (U_\setC v_1)^{[l]}
    \\
    =& L((\lambda^0,0)) \tensor \left(
      \mathcal{M}_{\setC}[t]\right)^{[l]}
  \end{align*}
  since $L((\lambda^0,0))$ is a simple $u_q$-module and since
  $\mathcal{M}_{\setC}[t]$ is a simple $U_\setC$-module.
\end{proof}

\begin{cor}
  \label{cor:4}
  $\left(\bigoplus_{\nu \in T^*} \left(
      f_\Sigma^\nu.L(\lambda^1)_{f_\Sigma} \right)^{[l]}\tensor
    L((\lambda^0,0)) \right)^{ss}\iso
  \mathcal{EXT}(L((\lambda^0,\lambda^1)))$.
\end{cor}
\begin{proof}
  This follows by the uniqueness of $\mathcal{EXT}(L(\lambda))$.
\end{proof}

\begin{cor}
  \label{cor:3}
  Let $L$ be an infinite dimensional admissible simple module. Then
  $\mathcal{EXT}(L)$ is of the form $\left(
    \left(\mathcal{M}\right)^{[l]} \tensor L((\lambda^0,0))
  \right)^{ss}$ for some $\mathfrak{g}$ coherent family $\mathcal{M}$
  (in the sense of~\cite{Mathieu}).
\end{cor}
\begin{proof}
  By Theorem~\ref{thm:contains_highest_weight} there exists a $w\in W$
  and a $\lambda\in X$ such that ${^w}\mathcal{EXT}(L)\iso
  \mathcal{EXT}(L(\lambda))$. By Corollary~\ref{cor:4}
  $\mathcal{EXT}(L(\lambda))\iso \left( \mathcal{M}\tensor
    L((\lambda^0,0))\right)^{ss}$ for some $\mathfrak{g}$ coherent
  family $\mathcal{M}$. By~\cite[Proposition~6.2]{Mathieu} and the
  fact that $L((\lambda^0,0))$ is finite dimensional we see that
  ${^{w}}\left( \mathcal{M}\tensor L((\lambda^0,0)) \right)^{ss}\iso
  \left( \mathcal{M} \tensor L((\lambda^0,0))\right)^{ss}$ for all
  $w\in W$.
\end{proof}

So in the root of unity case the classification of torsion free
modules reduces to the classification of classical torsion free
modules. By Proposition~\ref{prop:20} a torsion free module is a
submodule of a semisimple irreducible coherent family so the problem
reduces to classifying semisimple irreducible coherent families. By
Corollary~\ref{cor:3} the classification of these coherent families
reduces to the classification in the classical case.

\bibliography{lit} \bibliographystyle{amsalpha}


\end{document}